\documentclass[11pt]{article}
\usepackage[english]{babel}
\usepackage[T1]{fontenc}
\usepackage{times}
\usepackage{mathtools, bm}
\usepackage{amssymb, bm}
\usepackage{graphicx}
\usepackage{mathrsfs}
\usepackage{stmaryrd}
\usepackage{amsthm}
\usepackage{listings}
\usepackage{pict2e}
\usepackage{pgf, tikz}
\usepackage{dsfont}
\usepackage{algorithm2e}
\usepackage{algorithmic}
\usepackage{multicol}
\usepackage{hyperref}

\allowdisplaybreaks

\thicklines

{
      \theoremstyle{plain}
      \newtheorem{assumption}{Assumption}
  }

\newtheorem{theorem}{Theorem}

\newtheorem{corollary}{Corollary}
\newtheorem{lemma}{Lemma}
\newtheorem{remark}{Remark}

\numberwithin{equation}{section} 
\numberwithin{lemma}{section} 
\numberwithin{remark}{section} 
\numberwithin{example}{section}
\numberwithin{corollary}{section}
\numberwithin{proposition}{section}

\date{}
\author{Arnaud Guillin\footnote{Laboratoire de Math\'{e}matiques Blaise Pascal  - Universit\'{e} Clermont-Auvergne and Institut Universitaire de France. Email : arnaud.guillin[AT]uca.fr},$\ $ Pierre Le Bris\footnote{ Laboratoire Jacques-Louis Lions - Sorbonne Université. Email : pierre.lebris[AT]sorbonne-universite.fr}$\ $ and Pierre Monmarch\'{e}\footnote{ Laboratoire Jacques-Louis Lions - Sorbonne Université. Email : pierre.monmarche[AT]sorbonne-universite.fr}}
\title{On systems of particles in singular repulsive interaction in dimension one : log and Riesz gas}

\begin{document}

\textheight=21cm \textwidth=13cm
\maketitle

\begin{abstract}
In this article, we prove the first quantitative uniform in time propagation of chaos for a class of systems of particles in singular repulsive interaction in dimension one that contains the Dyson Brownian motion. We start by establishing existence and uniqueness for the Riesz gases, before proving propagation of chaos with an original approach to the problem, namely coupling with a Cauchy sequence type argument. We also give a general argument to turn a result of weak propagation of chaos into a strong and uniform in time result using the long time behavior and some bounds on moments, in particular enabling us to get a uniform in time version of the result of C\'epa-L\'epingle \cite{cepa_lepingle}.
\end{abstract}


%
%
%
%

\section{Introduction}

We consider the one dimensional N-particle system in mean field interaction
\begin{equation}\label{particle_system}
dX^i_t=\sqrt{2\sigma_N}dB^i_t-U'\left(X^i_t\right)dt-\frac{1}{N}\sum_{j\neq i}V'(X^i_t-X^j_t)dt.
\end{equation}
where for all $i\in\{1,...,N\}$, $X^i_t$ denotes the position in $\mathbb{R}$ of the i-th particle, $(B^i_t)_i$ are independent Brownian motions, and $\sigma_N$ is a diffusion coefficient that may depend on $N$. We denote $\textbf{X}^N_t=\left(X^1_t,...,X^N_t\right)$. We will refer to $U$ as the confining potential and $V$ as the interaction potential, on which we will specify the assumptions later. Finally, we denote by $\rho^N_t$ the law of $(X^1_t,...,X^N_t)$

The goal of this article is to give various results concerning equation \eqref{particle_system} in the case where $V$ is a singular repulsive interaction potential. The main motivating example is the (generalized)  Dyson Brownian motion
\begin{equation}\label{DBM}
dX^i_t=\sqrt{\frac{2\sigma}{N}}dB^i_t-\lambda X^i_tdt+\frac{1}{N}\sum_{j\neq i}\frac{1}{X^i_t-X^j_t}dt.
\end{equation}
Equation~\ref{DBM} is satisfied, for $\lambda=0$, by the eigenvalues of an $N\times N$ Hermitian matrix valued Brownian motion, as observed by Dyson in 1962 \cite{Dyson}. For \linebreak$\lambda>0$, it corresponds to the eigenvalues of an $N\times N$ Hermitian matrix valued Ornstein-Uhlenbeck process (see for instance \cite{Rogers_Shi}).

The work of Wigner \cite{Wigner} is often considered to be the starting point of Random Matrix Theory. The main observation is that, for a  Wigner matrix (a symmetric $N\times N$ matrix whose entries above the main diagonal are independent centered variables), the empirical distribution of the eigenvalues converges weakly as $N\rightarrow\infty$ to the standard semi-circle distribution. We refer to \cite{anderson_guionnet_zeitouni_2009} and references therein for a thorough introduction on Random Matrix Theory.

The main result of this article concerns the limit, as $N$ goes to infinity, of \eqref{particle_system}, which can be considered as a dynamical version of the convergence of the eigenvalues of a Wigner matrix. What we wish to prove is that \textit{in a system of $N$ particles in mean-field interaction, as $N$ goes to infinity, two particles become more and more statistically independent}. Kac \cite{kac1956} described this behavior as \textit{propagation of chaos}, and we refer to Sznitman \cite{Saint_Flour_1991} for a landmark study of the phenomenon. The notion of \textit{chaos} refers to the independence, while \textit{propagation} alludes to the fact that having this property of independence at the limit at time $0$ will be sufficient to ensure the same independence at later time $t$.

This limit for the Dyson Brownian motion was recently studied in \cite{BDLL21} using a notion of \textit{spectral dominance}, and obtained without convergence rate.

Throughout this article, we denote by $\mu^N_t:=\frac{1}{N}\sum_{i=1}^N\delta_{X^i_t}$ the empirical measure at time $t$ of the $N$ particle system. As proven by Sznitman \cite{Saint_Flour_1991}, the convergence of the empirical measure towards a constant random variable $\bar{\rho}_t$ is equivalent to the property of propagation of chaos. Very formally, this limit $\bar{\rho}_t$ is a weak solution to the non linear equation of \textit{McKean-Vlasov} type
\begin{equation}\label{NL}
\partial_t \bar{\rho}_t=\partial_x\left(\left(U'+V'\ast\mu_t\right)\mu_t\right)+\sigma \partial^2_{xx}\bar{\rho}_t,
\end{equation}
where $\sigma$ is the limit (possibly $0$) of $\sigma_N$ as $N\rightarrow\infty$ and $\ast$ is the (space) convolution operation. The stochastic differential equation associated to \eqref{NL} is
\begin{equation}\label{NL_SDE}
\left\{\begin{array}{ll}
dX_t=\sqrt{2\sigma}dB_t-U'(X_t)dt-V'\ast\bar{\rho}_t(X_t)dt,\\
\bar{\rho}_t=\text{Law}(X_t),
\end{array}
\right.
\end{equation}
and can also be seen as the formal limit of the stochastic differential equation (SDE) \eqref{particle_system}, noticing $\frac{1}{N}\sum_{j}V'(X^i_t-X^j_t)=V'\ast\mu^N_t(X^i_t).$ At this stage however, let us insist on the fact that the objects and solutions of \eqref{NL} and \eqref{NL_SDE} can be ill defined, especially when $V'$ is singular.

As we aim at deriving {\it quantitative} propagation of chaos result, we need a {\it distance}. For $\mu$ and $\nu$  two probability measures on $\mathbb{R}^{2}$, denote by $\Pi\left(\mu,\nu\right)$ the set of couplings of $\mu$ and $\nu$, i.e. the 
set of probability measures $\Gamma$ on $\mathbb{R}\times\mathbb{R}$ with $\Gamma(A\times \mathbb{R}) = \mu(A)$ and $\Gamma(\mathbb{R}\times 
A ) = \nu(A)$ for all Borel set $A$ of $\mathbb{R}$. We define the $L^p$ Wasserstein distance, with $p\geq1$, as  
\begin{align*}
\mathcal{W}_{p}\left(\mu,\nu\right)&=\left(\inf_{\Gamma\in\Pi\left(\mu,\nu\right)}\int|x-y|^p\Gamma\left(dxdy\right)\right)^{1/p}.
\end{align*}
It is important to notice (see for instance Remarks~3.28~and~3.30 of \cite{cuturi_peyre}) that in dimension 1 the optimal coupling (i.e the one realizing the infimum) for $\mathcal{W}_p$, $p\geq1$, is known as it is the monotone map. In particular, for two sets of points $(x_i)_{i\in\{1,...,N\}}$ and $(y_j)_{j\in\{1,...,N\}}$, assuming without loss of generality that\linebreak ${x_1\leq...\leq x_N}$ and $y_1\leq...\leq y_N$, and two measures $\mu=\frac{1}{N}\sum_i\delta_{x_i}$ and \linebreak${\nu=\frac{1}{N}\sum_j\delta_{y_j}}$, one has
\begin{align*}
\mathcal{W}_{p}\left(\mu,\nu\right)^p=\frac{1}{N}\sum_i|x_i-y_i|^p.
\end{align*}
There exists many ways of proving propagation of chaos, let us mention some.
\begin{itemize}
\item The main probabilistic tool, as used by McKean (see for instance \cite{McKean_original}) and then popularised by Sznitman \cite{Saint_Flour_1991}, is the coupling method. It consists in coupling the solution of \eqref{particle_system} with $N$ independent copies $(\bar{X}^i_t)_i$ of the solution \eqref{NL_SDE}. The goal is to control the Wasserstein distance, which by definition can be written as
\begin{align*}
\mathcal{W}_{d}\left(\rho^N_t,\bar{\rho}_t^{\otimes N}\right)&=\inf_{\Gamma\in\Pi\left(\rho^N_t,\bar{\rho}_t^{\otimes N}\right)}\mathbb{E}^\Gamma\left(\sum_{i=1}^Nd(X^i_t,\bar{X}^i_t)\right).
\end{align*}
Instead of considering the minimum over all couplings, the key idea is to construct a specific one, which will therefore provide an upper bound on the Wasserstein distance. Well known coupling methods include the \textit{synchronous} coupling \cite{Saint_Flour_1991, CGM08}, or the more recent \textit{reflection} coupling as suggested by Eberle \cite{Eberle_reflection, Eberle_Guillin_Zimmer_Harris, Uniform_Prop_Chaos}. The main benefit of this method of proof is that it allows for a better probabilistic understanding of the processes and gives quantitative speed of convergence in the case of Lipschitz continuous interactions. However, to the authors' knowledge, coupling methods have not yet given results in the case of singular interactions.
\item Using tools from PDE analysis, and functional inequalities, in order to show convergence of $\rho^N_t$ towards $\bar{\rho}_t^{\otimes N}$, recent progress have been made using a modulated energy \cite{Serfaty_Coulomb, serfaty_2020, rosenzweig2021globalintime}, by considering the relative entropy of $\rho^N_t$ with respect to $\bar{\rho}_t^{\otimes N}$ \cite{Jabin_Wang_2018} or by combining these two quantities into a modulated free energy \cite{BJW19}. These quantities have proven useful in showing propagation of chaos for systems of particle in singular interaction by making full use of the regularity and bounds on the moments of the limit equation \eqref{NL}.
\item Another method, that lies somewhere in between the fields of probability and PDE analysis, consists in proving the tightness or compactness of the set of empirical measure, showing that the limit of any convergent subsequence satisfies \eqref{NL}, and proving the uniqueness of the solution of \eqref{NL}. This has been for instance done for singular interaction kernels, in the specific case of \eqref{DBM} \cite{Rogers_Shi, cepa_lepingle, Li_Li_Xie}. This method, however, does not provide quantitative convergence rates.
\end{itemize}

Notice that all the methods described above rely on the properties of the limit equation \eqref{NL},  because one needs to either give sense to the quantity $V'\ast\mu_t$ in \eqref{NL_SDE} (and maybe show some properties in order to carry out computations) to use coupling methods, prove bounds and regularity on the solution in order to use PDE related methods, or at the very least prove the uniqueness of the solution of \eqref{NL}. This study of the limiting equation can be a quite challenging task. 

In this article, we describe a method that relies only on the well posedness of the system of particles \eqref{particle_system} and which provides a quantitative (and in some cases uniform in time) result of propagation of chaos. We make full use of the fact that in dimension one the particles will stay ordered, and that as a consequence the interaction we consider will be convex. Using a coupling method, we prove that by taking any independent sequence of empirical measures, it is a Cauchy sequence. Then, independence ensures the fact that the limit measure is an almost surely constant random variable. To the authors' knowledge, such a method has not been used before to prove propagation of chaos.

Let us now introduce our main assumptions. Consider $U\in\mathcal{C}^2(\mathbb{R})$, and make the following assumptions
\begin{assumption}\label{Hyp_U_lip}
$U'$ is Lipchitz continuous, i.e there exists $L_U$ such that for all $x\in\mathbb{R}$ we have $\left|U''(x)\right|\leq L_U$. This implies
\begin{equation*}
\forall x,y\in\mathbb{R},\ \left|U'(x)-U'(y)\right|\leq L_U\left|x-y\right|,
\end{equation*}
and
\begin{equation*}
\exists A>0,\ \forall x\in\mathbb{R},\ \left|U'(x)\right|\leq L_U| x|+A.
\end{equation*}
\end{assumption}
This first set of conditions will be used when establishing existence and uniqueness of solutions of \eqref{particle_system} as well as non uniform in time propagation of chaos. For further results, for simplicity, the study will be restricted to the quadratic case, namely:
\begin{assumption}\label{Hyp_U_conv}
There is $\lambda>0$ such that $U$ is explicitly given by
$$\forall x\in\mathbb{R},\ \ U(x)=\frac{\lambda}{2} x^2.$$
\end{assumption}
The condition on the interaction potential is the following:
\begin{assumption}\label{Hyp_V}
There exists $\alpha\geq1$ such that
\begin{equation}
 \forall x\in\mathbb{R}^*,\ V'(x)=-\frac{x}{|x|^{\alpha+1}}\,,
\end{equation}
and we thus consider
\begin{equation}
V(x)=
\left\{
\begin{array}{ll}
\frac{1}{\alpha-1}|x|^{-{\alpha+1}}&\text{ if }\alpha>1\\
-\ln\left(|x|\right)&\text{ if }\alpha=1
\end{array}
\right.
\end{equation}
Notice that for all $x\in\mathbb{R}^*$, $V'(x)=-V'(-x)$, and $V''(x)=\frac{\alpha}{|x|^{\alpha+1}}$.
\end{assumption}

Let us consider the open set
\begin{align*}
\mathcal{O}_N:=\left\{\textbf{X}=(x_1,...,x_N)\in\mathbb{R}^N\ \text{ s.t. }\ -\infty<x_1<...<x_N<\infty\right\}.
\end{align*}
We sum up the main results of the article in the following theorem.

\begin{theorem}\label{thm_resume}
{\rm A]} Under Assumption~\ref{Hyp_U_lip}~and~\ref{Hyp_V}, for $\alpha=1$ and $\sigma_N\leq\frac{1}{N}$ or for $\alpha>1$, there exists a unique strong solution to \eqref{particle_system}. 

{\rm B]} Under Assumptions~\ref{Hyp_U_conv}~and~\ref{Hyp_V}, denoting by $\rho^{1,N}_t$ and $\rho^{2,N}_t$ the probability densities on $\mathcal{O}_N$ of the particle systems with respective initial conditions $\rho^{1,N}_0$ and $\rho^{2,N}_0$, we have
\begin{equation*}
\forall t\geq0,\ \ \ \ \mathcal{W}_2\left(\rho^{1,N}_t,\rho^{2,N}_t\right)\leq e^{-\lambda t}\mathcal{W}_2\left(\rho^{1,N}_0,\rho^{2,N}_0\right).
\end{equation*}

{\rm C]} Still under Assumptions~\ref{Hyp_U_conv}~and~\ref{Hyp_V}, let $\mu^N_t:=\frac{1}{N}\sum_{i=1}^N\delta_{X^i_t}$ be the empirical measure at time $t$ of the solution of \eqref{particle_system}. Assume there exists $\bar{\rho}_0$ such that $\mathbb{E}\mathcal{W}^2_2(\mu^N_0,\bar{\rho}_0)\rightarrow 0$ as $N\rightarrow\infty$. For $\alpha\in[1,2[$ (with the additional assumption $\sigma_N\leq\frac{1}{N}$ for $\alpha=1$), there exist $(\rho_t)_{t\geq 0}\in\mathcal{C}(\mathbb{R}^+,\mathcal{P}_2(\mathbb{R}))$, as well as universal constants $C_1,C_2>0$ and a quantity $C_0^N>0$ that depends on the initial condition and such that $C_0^N\rightarrow0$ as $N\rightarrow\infty$, such that for all $N\geq1$ and all $t\geq0$
\begin{align*}
\mathbb{E}\left(\mathcal{W}_2(\mu^N_t,\bar{\rho}_t)^2\right)\leq e^{-2\lambda t}C_0^N+\frac{C_1}{N^{(2-\alpha)/\alpha}}+C_2\sigma_N,
\end{align*}
where $(\bar{\rho}_t)_t$ satisfies, for all functions $f\in\mathcal{C}^2(\mathbb{R})$ with bounded derivatives such that $f$, $f'$, $f'U'$, and $f''$ are Lipschitz continuous and that $f'U'$ is bounded, the following equation: $\forall t\geqslant 0$,
\begin{align*}
\int_{\mathbb{R}}f(x)\bar{\rho}_t(dx)=&\int_{\mathbb{R}}f(x)\bar{\rho}_0(dx)-\int_{0}^t\int_{\mathbb{R}}f'(x)U'(x)\bar{\rho}_s(dx)ds\\
&+\frac{1}{2}\int_{0}^t\int\int_{\{x\neq y\}}\frac{(f'(x)-f'(y))(x-y)}{|x-y|^{\alpha+1}}\bar{\rho}_s(dx)\bar{\rho}_s(dy)ds.
\end{align*}
\end{theorem}

Remark that the statement B], as well as functional inequalities such as Poincar\'e or logarithmic Sobolev inequalities, has been obtained in \cite{CL2020} in the case $\alpha=1$. Furthermore, statement C] extends the result of \cite{Berman_Onnheim}, in which similar systems are studied using the theory of gradient flows and (non uniform in time) propagation of chaos is obtained for $\alpha<2$ without convergence rate.

We split Theorem~\ref{thm_resume} above into the more precise Theorems~\ref{exist_unique},~\ref{long_time_beha},~\ref{thm_resume_prop}~and~\ref{id_limit} below. The organization of the article is as follows :
\begin{itemize}
\item In Section~\ref{sec_some_results} we prove various results concerning particle system \eqref{particle_system}.  In Section~\ref{sec_exist}, we show that, for $\alpha>1$ with any diffusion coefficient $\sigma_N$ or $\alpha=1$ with $\sigma_N\leq\frac{1}{N}$, there exists a unique strong solution to \eqref{particle_system} under Assumptions~\ref{Hyp_U_lip} and \ref{Hyp_V}. Furthermore, the particles stay in the same order at all time. See Theorem~\ref{exist_unique}. In Section~\ref{sec_long_time}, we show the long time convergence of the particle system under Assumptions~\ref{Hyp_U_conv} and \ref{Hyp_V}. See Theorem~\ref{long_time_beha}. In Section~\ref{sec_bornes}, we prove bounds on the expectation of interaction that will be  useful later.
\item Section~\ref{sec_prop} contains the main proofs of the article concerning the propagation of chaos for \eqref{particle_system} in the case $\sigma_N\rightarrow0$. For clarity, we separate the case $\alpha=1$ (in Section~\ref{sec_prop_alpha_1}) and $\alpha\in]1,2[$ (in Section~\ref{sec_prop_alpha_sup_1}), as the former allows for a proof that contains all the ideas with little technical difficulties, while the latter requires the more precise bounds obtained in Section~\ref{sec_bornes}. See Theorem~\ref{thm_resume_prop}.
\item In Section~\ref{appendice_id_limit}, we identify, in a more rigorous way than the calculations above, the equation satisfied by the limit $\bar{\rho}_t$. See Theorem~\ref{id_limit}. In particular, we highlight an argument which intuitively suggests that  $\alpha=2$ should be the critical case for the well-posedness of the limit.
\end{itemize}

Finally, in Section~\ref{sec_from_weak_to_strong}, we show how one can turn a result of weak propagation of chaos, such as the one obtained in \cite{Rogers_Shi, cepa_lepingle, Li_Li_Xie}, into a strong and uniform in time result by using the long time convergence and some bounds on the moments of the particle system. See Theorem~\ref{unif_prop_chaos_a_partir_de_tout_un_tas_de_proprietes_que_je_ne_vais_pas_donner_parce_que_la_reference_devient_trop_longue_mais_vous_voyez_l_idee}. This yields in particular strong uniform in time propagation of chaos in the case of $\alpha=1$ and constant diffusion coefficient $\sigma_N=\sigma\neq0$  that Theorem~\ref{thm_resume_prop} cannot deal with, though without a quantitative rate of convergence, using the result of \cite{cepa_lepingle}. 

\begin{corollary}
Under Assumptions~\ref{Hyp_U_conv}~and~\ref{Hyp_V}, for $\alpha=1$, $\sigma_N= \sigma\in\mathbb{R}$, assume we have for all $N$ an initial condition $(X^1_0,...,X^N_0)$ with bounded forth moments (i.e $\mathbb{E}\left(\frac{1}{N}\sum_{i=1}^N|X^i_0|^4\right)$) such that ${t\mapsto\mathbb{E}\left(\frac{1}{N}\sum_{i=1}^N|X^i_t-X^i_0|^2\right)}$ is continuous in $t=0$ uniformly in $N$. Then  we get strong uniform in time propagation of chaos, i.e
\begin{align*}
\forall \epsilon>0,\ \exists N\geq0,\ \forall t\geq0,\ \forall n\geq N,\ \mathbb{E}\left(\mathcal{W}_2\left(\mu^n_t,\bar{\rho}_t\right)\right)<\epsilon.
\end{align*}
\end{corollary}

We give in Appendix~\ref{cont_0_alpha_1} some sufficient conditions for this assumption of continuity for the initial conditions.

%
%
%
%

\section*{Notations}

We try to keep coherent notations throughout the article, but as the various objects and what they represent may become confusing, we list them here for reference :
\begin{itemize}
\item $\mathcal{P}(\mathcal{X})$ is the set of probability measures on the set $\mathcal{X}$, and  $\mathcal{P}_p(\mathcal{X})$ is the set of probability measures on the set $\mathcal{X}$ with finite p-th moment,
\item $(X^1_t,...,X^N_t)$, or $(X^{1,N}_t,...,X^{N,N}_t)$ when we need to insist on the total number of particles, is the solution of the SDE defining our particle system. $X^i_t$ denotes the position in $\mathbb{R}$ of the i-th particle,
\item $\mathcal{O}_N:=\left\{\textbf{X}=(x_1,...,x_N)\in\mathbb{R}^N\ \text{ s.t. }\ -\infty<x_1<...<x_N<\infty\right\}$ is the set in which we prove the solutions are,
\item $\mu^N_t:=\frac{1}{N}\sum_{i=1}^N\delta_{X^i_t}\in\mathcal{P}(\mathbb{R})$ is the empirical measure at time $t$ of the $N$ particle system. Notice that it is a random variable on the set $\mathcal{P}(\mathbb{R})$,
\item $\xi^N_t\in\mathcal{P}(\mathcal{P}(\mathbb{R}))$ is the law of $\mu^N_t$,
\item $\rho^N_t\in\mathcal{P}(\mathcal{O}_N)$ is the joint law of $(X^1_t,...,X^N_t)$,
\item $\bar{\rho}_t\in\mathcal{P}(\mathbb{R})$ is the limit towards which $\mu^N_t$ will converge,
\item all the notations above used with $t=\infty$ refer to the stationary distribution (provided it exists),
\item $\mathcal{C}(\mathbb{R}^+,\mathcal{P}_2(\mathbb{R}))$ is the space of continuous functions taking values in the space of probability measures $\mathcal{P}_2(\mathbb{R})$ endowed with the $L^2$ Wasserstein distance,
\item $\mu^N=(\mu^N_t)_{t\geq0}\in\mathcal{C}(\mathbb{R}^+,\mathcal{P}_2(\mathbb{R}))$ and $\bar{\rho}=(\bar{\rho}_t)_{t\geq0}\in\mathcal{C}(\mathbb{R}^+,\mathcal{P}_2(\mathbb{R}))$,
\item for a probability measure $\mu$ and a measurable function $f$, we may denote both $\mu(f):=\int fd\mu$ and ${\mathbb{E}^{\mu}(f(X)):=\int fd\mu}$. 
\end{itemize}

%
%
%
%

\section{Some results on the particle system}\label{sec_some_results}

%
%

\subsection{Existence, uniqueness and no collisions}\label{sec_exist}

The goal of this subsection is to prove the following result.


\begin{theorem}\label{exist_unique}
Consider $N\geq 2$, and $-\infty<x_1<...<x_N<\infty$. Under Assumptions~\ref{Hyp_U_lip}~and~\ref{Hyp_V} :
\begin{itemize}
\item If $\alpha>1$, for any $\sigma_N\geq0$, there exists a unique strong solution \linebreak $X=(X^1,...,X^N)$ to the stochastic differential equation \eqref{particle_system} with initial condition $X^1_0=x_1$, ..., $X^N_0=x_N$, which furthermore satisfies $X_t\in\mathcal{O}_N$ for all $t\geq0$, $\mathbb{P}$-a.s.
\item The same result holds for $\alpha=1$ and  $\sigma_N\leq\frac{1}{N}$.
\end{itemize}
\end{theorem}

\begin{remark}
In the case $\alpha=1$ and $\sigma_N=\sigma>0$, the existence of a unique strong solution has been written in Theorem~2.5 of \cite{cepa_lepingle}, where the authors allow collisions between particles and show that the system still satisfies $X_t\in\overline{\mathcal{O}_N}$ i.e $$-\infty<X^1_t\leq...\leq X^N_t<\infty\ \ \text{ for all }t\geq0, \mathbb{P}\text{-a.s.}$$
In the case $\alpha=1$ and $\sigma_N\leq\frac{1}{N}$, the proof of existence, uniqueness and absence of collision has been done in \cite{Rogers_Shi} or in the more recent \cite{Li_Li_Xie}. For the sake of completeness, and because it uses similar calculations, we also write the proof in this case here.
\end{remark}

Denote the infinitesimal generator
\begin{align*}
\mathcal{L}^{N,\alpha}f(x_1,...,x_N):=&-\sum_{i=1}^NU'(x_i)\partial_if(x_1,...,x_N)\\
&-\frac{1}{N}\sum_{i\neq j}V'(x_i-x_j)\partial_if(x_1,...,x_N)+\sigma_N\Delta f(x_1,...,x_N).
\end{align*}
and consider, for $\textbf{X}=(x_1,...,x_N)\in\mathbb{R}^N$
\begin{align*}
U_{int, \alpha}(\textbf{X}):=\frac{1}{2N}\sum_{i\neq j}V(x_i-x_j),\ \ \ \ \ H_{\alpha}(\textbf{X}):=U_{int, \alpha}(\textbf{X})+\sum_{i=1}^N \frac{x^2_i}{2},
\end{align*}
where $U_{int}$ denotes the interaction potential. We prove the following lemma


\begin{lemma}\label{lya_alpha}
Let $N>1$. Under Assumptions~\ref{Hyp_U_lip}~and~\ref{Hyp_V}, for $\alpha>1$, there exists $C^{N,\alpha},D^{N,\alpha}>0$ such that for all $\textbf{X}\in\mathcal{O}_N$
\begin{align*}
\mathcal{L}^{N,\alpha} H_{\alpha}(\textbf{X})\leq D^{N,\alpha}+C^{N,\alpha}H(\textbf{X}).
\end{align*}
Under the additional Assumption~\ref{Hyp_U_conv}, still for $\alpha>1$, there exists $C^{N,\alpha},D^{N,\alpha}>0$ such that for all $\textbf{X}\in\mathcal{O}_N$
\begin{align*}
\mathcal{L}^{N,\alpha} H_{\alpha}(\textbf{X})\leq D^{N,\alpha}-C^{N,\alpha}H(\textbf{X}).
\end{align*}
\end{lemma}

This lemma shows that, for a force $U'$ Lipschitz continuous, the energy does not explode in finite time, which will provide us the existence of the solution of \eqref{particle_system}, and the absence of collision between particles. For a potential $U$ convex, it even yields a uniform in time bound on the second moment of the particles, and on the expectation of $\frac{1}{\left|X^i-X^j\right|^{\alpha-1}}$ (even though we do not use this result to bound these moments, as $D^{N,\alpha}$ and $C^{N,\alpha}$ depend rather badly on $N$).


\begin{proof}
We compute
\begin{align*}
\mathcal{L}^{N,\alpha} H_{\alpha}(\textbf{X})=&-\sum_{i=1}^N\left(U'(x_i)+\frac{1}{N}\sum_{j\neq i}V'(x_i-x_j)\right)\left(x_i+\frac{1}{N}\sum_{j\neq i}V'(x_i-x_j)\right)\\
&+\sigma_N\sum_{i=1}^N1+\frac{\sigma_N}{N}\sum_{i=1}^N\sum_{j\neq i}V''(x_i-x_j)\\
=&-\sum_{i=1}^N U'(x_i)x_i-\frac{1}{N}\sum_{i=1}^N\sum_{j\neq i}U'(x_i)V'(x_i-x_j)\\
&-\frac{1}{N}\sum_{i=1}^N\sum_{j\neq i}x_iV'(x_i-x_j)+N\sigma_N\\
&-\sum_{i=1}^N\left(\frac{1}{N}\sum_{j\neq i}V'(x_i-x_j)\right)^2+\frac{\sigma_N}{N}\sum_{i=1}^N\sum_{j\neq i}V''(x_i-x_j).
\end{align*}
We have, under Assumptions~\ref{Hyp_U_lip}~and~\ref{Hyp_V}
\begin{align}
-\sum_{i=1}^N U'(x_i)x_i\leq\sum_{i=1}^NL_U|x_i|^2&+A|x_i|\leq\left(L_U+\frac{1}{2}\right)\sum_{i=1}^N|x_i|^2+\frac{NA^2}{2}\label{ineq_lya_U}\\
-\frac{1}{N}\sum_{i=1}^N\sum_{j\neq i}U'(x_i)V'(x_i-x_j)=&-\frac{1}{N}\sum_{j< i}\left(U'(x_i)-U'(x_j)\right)\frac{x_i-x_j}{\left|x_i-x_j\right|^{\alpha+1}}\nonumber\\
\leq&\frac{L_U}{N}\sum_{j< i}\frac{1}{\left|x_i-x_j\right|^{\alpha-1}}\label{ineq_lya_V},\\
-\frac{1}{N}\sum_{i=1}^N\sum_{j\neq i}x_iV'(x_i-x_j)\leq&\frac{1}{N}\sum_{j< i}\frac{1}{\left|x_i-x_j\right|^{\alpha-1}}\label{ineq_lya_V_2}.
\end{align}
Let us now consider $\left|\nabla U_{int, \alpha}(\textbf{X})\right|=\left(\sum_{i=1}^N\left(\frac{1}{N}\sum_{j\neq i}V'(x_i-x_j)\right)^2\right)^{1/2}$. We follow the proof of Lemma~5.15 of \cite{herzog_mattingly}. Let $j<i$, which implies $x_j<x_i$, and denote
\begin{align*}
 \xi_k\left(\textbf{X}\right)=&\left\{\begin{array}{ll}1&\text{ if }x_k<x_i\\
-1 &\text{ otherwise.}\end{array}\right.
\end{align*}
Then, considering $\xi(\textbf{X})=\left(\xi_1(\textbf{X}),...,\xi_N(\textbf{X})\right)$, we have
\begin{align*}
\sqrt{N}\left|\nabla U_{int, \alpha}(\textbf{X})\right|\geq& \xi(\textbf{X})\cdot\nabla U_{int}(\textbf{X})\\
=&\frac{1}{N}\sum_k\xi_k(\textbf{X})\sum_{l\neq k}V'(x_k-x_l)\\
=&-\frac{1}{N}\sum_{k<l}\left(\xi_k(\textbf{X})-\xi_l(\textbf{X})\right)\frac{x_k-x_l}{\left|x_k-x_l\right|^{\alpha+1}}.
\end{align*}
Notice that, for $k<l$, $\xi_k(\textbf{X})\neq\xi_l(\textbf{X})$ if and only if $x_k<x_i\leq x_l$, in which case we have $\xi_k(\textbf{X})-\xi_l(\textbf{X})=2$ and $\left(x_k-x_l\right)<0$. Therefore, the sum above only contains nonpositive terms. In particular, choosing $k=j$ and $l=i$, we get
\begin{align*}
\sqrt{N}\left|\nabla U_{int, \alpha}(\textbf{X})\right|\geq&\frac{2}{N}\frac{1}{\left|x_i-x_j\right|^{\alpha}}.
\end{align*}
This holds for any $j<i$, thus 
\begin{align*}
\sqrt{N}\frac{N(N-1)}{2}\left|\nabla U_{int, \alpha}(\textbf{X})\right|\geq&\frac{2}{N}\sum_{j<i}\frac{1}{\left|x_i-x_j\right|^{\alpha}},\ \ \ \ \text{ i.e }\\
\left|\nabla U_{int, \alpha}(\textbf{X})\right|\geq&\frac{4}{N^2(N-1)\sqrt{N}}\sum_{j<i}\frac{1}{\left|x_i-x_j\right|^{\alpha}}.
\end{align*}
We therefore have
\begin{align*}
-\sum_{i=1}^N&\left(\frac{1}{N}\sum_{j\neq i}V'(x_i-x_j)\right)^2+\frac{\sigma_N}{N}\sum_{i=1}^N\sum_{j\neq i}V''(x_i-x_j)\\
&\hspace{0.5cm}\leq\sum_{j< i}\frac{\sigma_N}{N}\frac{2\alpha}{\left|x_i-x_j\right|^{\alpha+1}}-\left(\frac{4}{N^2(N-1)\sqrt{N}}\right)^2\frac{1}{\left|x_i-x_j\right|^{2\alpha}}.
\end{align*}
For $\alpha>1$, thanks to Young's inequality, there is a constant $C_N$ such that $$\frac{2\sigma_N\alpha}{N}\frac{1}{\left|x_i-x_j\right|^{\alpha+1}}-\left(\frac{4}{N^2(N-1)\sqrt{N}}\right)^2\frac{1}{\left|x_i-x_j\right|^{2\alpha}}<C_N.$$ Therefore, using this result along with \eqref{ineq_lya_U} and \eqref{ineq_lya_V}, we prove the existence of two nonnegative constants $C$ and $D$, possibly depending on $N$, such that 
\begin{align*}
\mathcal{L}^{N,\alpha} H_{\alpha}(\textbf{X})\leq D+CH_{\alpha}(\textbf{X}).
\end{align*}
Let us now modify the various estimates under the additional Assumption~\ref{Hyp_U_conv}. We may replace the control \eqref{ineq_lya_U} by
\begin{align}
-\sum_{i=1}^N U'(x_i)x_i=-2\lambda\sum_{i=1}^N\frac{x_i^2}{2}\label{ineq_lya_U_2}
\end{align}
Then, instead of \eqref{ineq_lya_V} and \eqref{ineq_lya_V_2}, we use the fact that there are $C_N$ and $D_N$ such that
\begin{align*}
\frac{(L_U+1)}{N}\frac{1}{\left|x_i-x_j\right|^{\alpha-1}} &+\frac{2\sigma_N\alpha}{N}\frac{1}{\left|x_i-x_j\right|^{\alpha+1}}-\left(\frac{4}{N^2(N-1)\sqrt{N}}\right)^2\frac{1}{\left|x_i-x_j\right|^{2\alpha}}\\
&<C_N-\frac{D_N}{\left|x_i-x_j\right|^{\alpha-1}}.
\end{align*}
Combining this inequality with \eqref{ineq_lya_U_2}, we prove the existence of two nonnegative constants $C$ and $D$, possibly depending on $N$, such that 
\begin{align*}
\mathcal{L}^{N,\alpha} H_{\alpha}(\textbf{X})\leq D-CH_{\alpha}(\textbf{X}).
\end{align*}
\end{proof}


\begin{lemma}\label{lya_alpha_1}
Let $N>1$. Under Assumptions~\ref{Hyp_U_lip}~and~\ref{Hyp_V}, for  $\alpha=1$ and $\sigma_N\leq\frac{1}{N}$, there exists $C^{N,\alpha}, D^{N,\alpha}>0$ such that for all $\textbf{X}\in\mathcal{O}_N$
\begin{align*}
\mathcal{L}^{N,\alpha} H_{\alpha}(\textbf{X})\leq D^{N,\alpha}+C^{N,\alpha}H(\textbf{X}).
\end{align*}
\end{lemma}


\begin{proof}
We compute
\begin{align*}
\mathcal{L}^{N,\alpha} H_{\alpha}(\textbf{X})=&-\sum_{i=1}^N U'(x_i)x_i-\frac{1}{N}\sum_{i=1}^N\sum_{j\neq i}U'(x_i)V'(x_i-x_j)\\
&-\frac{1}{N}\sum_{i=1}^N\sum_{j\neq i}x_iV'(x_i-x_j)+N\sigma_N\\
&-\sum_{i=1}^N\left(\frac{1}{N}\sum_{j\neq i}V'(x_i-x_j)\right)^2+\frac{\sigma_N}{N}\sum_{i=1}^N\sum_{j\neq i}V''(x_i-x_j).
\end{align*}
Let us consider
\begin{align*}
\sum_{i=1}^N\left(\frac{1}{N}\sum_{j\neq i}\frac{x_i-x_j}{\left|x_i-x_j\right|^2}\right)^2=&\frac{1}{N^2}\sum_{i=1}^N\sum_{j,k\neq i}\frac{x_i-x_j}{\left|x_i-x_j\right|^2}\frac{x_i-x_k}{\left|x_i-x_k\right|^2}\\
=&\frac{1}{N^2}\sum_{i=1}^N\sum_{j\neq i}\frac{1}{\left|x_i-x_j\right|^2}\\
&+\frac{1}{N^2}\sum_{i,j,k\text{ distincts}}\frac{x_i-x_j}{\left|x_i-x_j\right|^2}\frac{x_i-x_k}{\left|x_i-x_k\right|^2},
\end{align*}
and using the fact that for $\textbf{X}\in\mathcal{O}_N$ and $i<j$ we have $x_i<x_j$, we obtain
\begin{align*}
\sum_{i,j,k\text{ distincts}}&\frac{x_i-x_j}{\left|x_i-x_j\right|^2}\frac{x_i-x_k}{\left|x_i-x_k\right|^2}\\
=&2\sum_{i=1}^N\sum_{\scriptsize{\begin{array}{ll}j<k\\ j,k\neq i\end{array}}}\frac{x_i-x_j}{\left|x_i-x_j\right|^2}\frac{x_i-x_k}{\left|x_i-x_k\right|^2}\\
=&2\sum_{i<j<k}\left(\frac{x_i-x_j}{\left|x_i-x_j\right|^2}\frac{x_i-x_k}{\left|x_i-x_k\right|^2}+\frac{x_j-x_i}{\left|x_j-x_i\right|^2}\frac{x_j-x_k}{\left|x_j-x_k\right|^2}\right.\\
&\left.\hspace{2cm}+\frac{x_k-x_j}{\left|x_k-x_j\right|^2}\frac{x_k-x_i}{\left|x_k-x_i\right|^2}\right)\\
=&2\sum_{i<j<k}\left(\frac{1}{x_j-x_i}\frac{1}{x_k-x_i}-\frac{1}{x_j-x_i}\frac{1}{x_k-x_j}+\frac{1}{x_k-x_i}\frac{1}{x_k-x_j}\right)\\
=&2\sum_{i<j<k}\frac{1}{x_j-x_i}\frac{1}{x_k-x_i}\frac{1}{x_k-x_j}\left(x_k-x_j-x_k+x_i+x_j-x_i\right)\\
=&0.
\end{align*}
Futhermore, the estimates \eqref{ineq_lya_U} and \eqref{ineq_lya_V} still hold. We thus have
\begin{align*}
-\sum_{i=1}^N U'(x_i)x_i\leq\left(L_U+\frac{1}{2}\right)\sum_{i=1}^N|x_i|^2+\frac{NA^2}{2},\\
-\frac{1}{N}\sum_{i=1}^N\sum_{j\neq i}(U'(x_i)+x_i)V'(x_i-x_j)\leq& \frac{1}{2}(L_U+1)(N-1).
\end{align*}
Therefore
\begin{align*}
\mathcal{L}^{N,\alpha} H_\alpha(\textbf{X})\leq&\frac{(L_U+1)(N-1)}{2}+\frac{NA^2}{2}+N\sigma_N+\left(L_U+\frac{1}{2}\right)\sum_{i=1}^N|x_i|^2\\
&+2\sum_{i<j}\left(\frac{\sigma_N}{N}-\frac{1}{N^2}\right)\frac{1}{\left|x_i-x_j\right|^2}.
\end{align*}
Noticing that there exist constants $C,D$ such that $\sum_{i=1}^N|x_i|^2\leq CH_{\alpha}(\textbf{X})+D$, we obtain the result if $\frac{\sigma_N}{N}\leq\frac{1}{N^2}$.
\end{proof}

\begin{proof}[Proof of Theorem~\ref{exist_unique}] 
For $R>0$, define $\tau_R:=\inf\{t\geq 0\ \ \text{ s.t }\ \ H_\alpha(\textbf{X}_t)>R\}$, $\tau:=\lim_{R\rightarrow\infty}\tau_R$, and $\tau_{\partial\mathcal{O}_N}:=\inf\{t\geq 0\ \ \text{ s.t }\ \ \textbf{X}_t\in\partial\mathcal{O}_N\}$. We have \linebreak$\{\tau=\infty\}\subset\{\tau_{\partial\mathcal{O}_N}=\infty\}.$ Equation \eqref{particle_system} with initial condition $\textbf{X}_0=\textbf{x}\in\mathcal{O}_N$ has a strong solution up to the stopping time $\tau$. Let us show that $\mathbb{P}_\textbf{x}(\tau=\infty)=1.$

\paragraph{$\bullet\ \ \ \alpha>1$ :} Itô's formula for the function $f(t,\textbf{x})=e^{-C^{N,\alpha}t}H_\alpha(\textbf{x})$, using Lemma~\ref{lya_alpha}, yields for all $R>0$ and $t\geq0$
\begin{align*}
\mathbb{E}_{\textbf{x}}\left(e^{-C^{N,\alpha}(t\wedge\tau_R)}H_\alpha(\textbf{X}_{t\wedge\tau_R})\right)\leq H_\alpha(\textbf{x})+\frac{D^{N,\alpha}}{C^{N,\alpha}},
\end{align*}
and thus, as $H_\alpha\geq0$,
\begin{align*}
Re^{-C^{N,\alpha}t}\mathbb{P}_x(\tau_R\leq t)\leq H_\alpha(\textbf{x})+\frac{D^{N,\alpha}}{C^{N,\alpha}}.
\end{align*}
We obtain, for all $t\geq0$
\begin{align*}
\mathbb{P}_x(\tau\leq t)=\lim_{R\rightarrow\infty}\mathbb{P}_x(\tau_R\leq t)\leq\lim_{R\rightarrow\infty}\frac{ H_\alpha(\textbf{x})+\frac{D^{N,\alpha}}{C^{N,\alpha}}}{R}e^{C^{N,\alpha}t}=0.
\end{align*}

\paragraph{$\bullet\ \ \ \alpha=1$ :} There exists a constant $H_0\in\mathbb{R}$, possibly depending on $N$, such that for all $\textbf{x}\in\mathcal{O}_N$, $H_\alpha(\textbf{x})\geq H_0$. Considering Itô's formula for the function
$f(t,\textbf{x})=e^{-C^{N,\alpha}t}\left(H_\alpha(\textbf{x})+H_0\right)$ , using Lemma~\ref{lya_alpha_1}, yields for all $R>0$ and $t\geq0$
\begin{align*}
\mathbb{E}_\textbf{x}\left(e^{-C^{N,\alpha}(t\wedge\tau_R)}\left(H_\alpha(\textbf{X}_{t\wedge\tau_R})+H_0\right)\right)\leq H_\alpha(\textbf{x})+H_0+\frac{D^{N,\alpha}}{C^{N,\alpha}},
\end{align*}
and thus, as $H_\alpha+H_0\geq0$,
\begin{align*}
e^{-C^{N,\alpha}t}(R+H_0)\mathbb{P}_x(\tau_R\leq t)\leq H_\alpha(\textbf{x})+H_0+\frac{D^{N,\alpha}}{C^{N,\alpha}}.
\end{align*}
We obtain, for all $t\geq0$
\begin{align*}
\mathbb{P}_x(\tau\leq t)=\lim_{R\rightarrow\infty}\mathbb{P}_x(\tau_R\leq t)\leq\lim_{R\rightarrow\infty}\frac{H_\alpha(\textbf{x})+H_0+\frac{D^{N,\alpha}}{C^{N,\alpha}}}{R+H_0}e^{C^{N,\alpha}t}=0.
\end{align*}

\paragraph{} We thus have, in both cases, $\forall t\geq0, \mathbb{P}_x(\tau> t)=1$. This implies the particle system almost surely does not explode nor collide in finite time.

Uniqueness of the solution of \eqref{particle_system} is a direct consequence of \eqref{unicite_en_vrai} in Theorem~\ref{long_time_beha} below. 
\end{proof}

%
%

\subsection{Long time behavior}\label{sec_long_time}

\begin{theorem}\label{long_time_beha}
Consider two solutions $X$ and $Y$ of \eqref{particle_system} driven by the same Brownian motions. Under Assumptions~\ref{Hyp_U_lip}~and~\ref{Hyp_V}, we have 
\begin{equation}\label{unicite_en_vrai}
\sum_{i=1}^N\left(X^i_t-Y^i_t\right)^2\leq e^{2L_U t}\sum_{i=1}^N\left(X^i_0-Y^i_0\right)^2.
\end{equation}
This yields strong uniqueness of the solution of\eqref{particle_system}. And, under Assumptions~\ref{Hyp_U_conv}~and~\ref{Hyp_V}, denoting by $\rho^{1,N}_t$ and $\rho^{2,N}_t$ the laws on $\mathcal{O}_N$ of the particle systems with respective initial conditions $\rho^{1,N}_0$ and $\rho^{2,N}_0$, we have
\begin{equation}\label{long_time}
\forall t\geq0,\ \ \ \ \mathcal{W}_2\left(\rho^{1,N}_t,\rho^{2,N}_t\right)\leq e^{-\lambda t}\mathcal{W}_2\left(\rho^{1,N}_0,\rho^{2,N}_0\right).
\end{equation}
Under Assumptions~\ref{Hyp_U_lip}~and~\ref{Hyp_V}, we have
\begin{equation}\label{unicite}
\forall t\geq0,\ \ \ \ \mathcal{W}_2\left(\rho^{1,N}_t,\rho^{2,N}_t\right)\leq e^{L_U t}\mathcal{W}_2\left(\rho^{1,N}_0,\rho^{2,N}_0\right).
\end{equation}
\end{theorem}


\begin{proof}

Let $(X^i_t)_{1\leq i\leq N}$ and $(Y^i_t)_{1\leq i\leq N}$ be two solutions of \eqref{particle_system} driven by the same set of Brownian motions (i.e coupled using a synchronous coupling), such that $X^1_t<...<X^N_t$ and $Y^1_t<...<Y^N_t$. Using Itô's formula, we have under Assumption~\ref{Hyp_U_lip}
\begin{align*}
d\left(\sum_{i=1}^N\left(X^i_t\right.\right.&\left.\left.-Y^i_t\right)^2\right)=-2\lambda\left(\sum_{i=1}^N\left(X^i_t-Y^i_t\right)^2\right)dt\\
&-\frac{1}{N}\sum_{i=1}^N2\left(X^i_t-Y^i_t\right)\sum_{j\neq i}\left(V'(X^i_t-X^j_t)-V'(Y^i_t-Y^j_t)\right)dt,
\end{align*}
with, since $x\rightarrow V'(x)$ is odd and increasing for $x>0$
\begin{align*}
\frac{2}{N}&\sum_{1\leq j\neq i\leq N}\left(X^i_t-Y^i_t\right)\left(V'(X^i_t-X^j_t)-V'(Y^i_t-Y^j_t)\right)\\
=&\frac{2}{N}\sum_{1\leq j<i\leq N}\left(\left(X^i_t-Y^i_t\right)-\left(X^j_t-Y^j_t\right)\right)\left(V'(X^i_t-X^j_t)-V'(Y^i_t-Y^j_t)\right)\\
=&\frac{2}{N}\sum_{1\leq j<i\leq N}\left(\left(X^i_t-X^j_t\right)-\left(Y^i_t-Y^j_t\right)\right)\left(V'(X^i_t-X^j_t)-V'(Y^i_t-Y^j_t)\right)\\
\geq&0.
\end{align*}
This implies
\begin{align*}
\frac{d}{dt}\sum_{i=1}^N\left(X^i_t-Y^i_t\right)^2\leq-2\lambda\sum_{i=1}^N \left(X^i_t-Y^i_t\right)^2
\end{align*}
i.e
\begin{align*}
d\left(e^{2\lambda t}\sum_{i=1}^N\left(X^i_t-Y^i_t\right)^2\right)=K_tdt,
\end{align*}
with $K_t\leq0$. We thus obtain
\begin{align*}
\sum_{i=1}^N\left(X^i_t-Y^i_t\right)^2\leq e^{-2\lambda t}\sum_{i=1}^N\left(X^i_0-Y^i_0\right)^2.
\end{align*}
This yields the result \eqref{long_time}.

In the case of Assumption~\ref{Hyp_U_lip}, similar calculations yield
\begin{align*}
\frac{d}{dt}\sum_{i=1}^N\left(X^i_t-Y^i_t\right)^2\leq 2L_U\sum_{i=1}^N \left(X^i_t-Y^i_t\right)^2,
\end{align*}
and thus \eqref{unicite}.
\end{proof}

%
%

\subsection{Some bounds}\label{sec_bornes}

The aim of this section is to provide some explicit bounds on the second moment of the empirical measure, as well as on the expectation of the interaction potential. These bounds will be useful later when proving propagation of chaos. Let, for $\textbf{x}\in\mathbb{R}^N$,
\begin{equation}\label{def_H_cal}
\mathcal{H}(\textbf{x})=\sum_{i=1}^N\left|x_i\right|^2-\frac{1}{2N}\sum_{i\neq j}\left|x_i-x_j\right|.
\end{equation}
The idea of considering this function comes from \cite{lu_mattingly}.


\begin{lemma}
Consider Assumptions~\ref{Hyp_U_conv} and \ref{Hyp_V}. The function $\mathcal{H}$ satisfies
\begin{align}\label{borne_H}
\forall \textbf{x}\in\mathbb{R}^N,\ \ \mathcal{H}(x)\geq \frac{1}{2}\sum_{i}|x_i|^2-N.
\end{align}
Given $(\textbf{X}_t)_t\geq0$ a solution of \eqref{particle_system}, we have the uniform in time bound 
\begin{align}\label{unif_tps_H}
\mathbb{E}\mathcal{H}(\textbf{X}_t)\leq e^{-2\lambda t}\mathbb{E}\mathcal{H}(\textbf{X}_0)+\frac{1}{\lambda }\left(N\sigma_N+\frac{C(\alpha,N)}{\alpha}\right),
\end{align}
as well as the following estimates 
\begin{align}
\mathbb{E}&\left(\int_0^t\frac{e^{2\lambda s}}{N^2}\sum_{i>j}\frac{i-j}{\left|X^i_s-X^j_s\right|^{\alpha}}ds\right)\nonumber\\
&\leq\frac{\alpha}{2}\left(\mathbb{E}\mathcal{H}(\textbf{X}_0)+Ne^{2\lambda t}+2N\sigma_N \frac{e^{2\lambda t}-1}{2\lambda}\right)+C(\alpha,N)\frac{e^{2\lambda t}-1}{2\lambda}\label{borne_interaction},\\
\mathbb{E}&\left(\int_0^t\frac{1}{N^2}\sum_{i>j}\frac{i-j}{\left|X^i_s-X^j_s\right|^{\alpha}}ds\right)\nonumber\\
&\leq\frac{\alpha}{2}\left(\mathbb{E}\mathcal{H}(\textbf{X}_0)+N+\left(2N\sigma_N+2\lambda N\right)t\right)+C(\alpha,N)t\label{borne_interaction_sup},
\end{align}
where
\begin{align*}
C(\alpha,N)=\left\{
\begin{array}{ll}
\frac{N-1}{2}&\text{ if }\alpha=1,\\
\frac{N}{2-\alpha}&\text{ if }\alpha\in]1,2[,\\
2N\ln N&\text{ if }\alpha=2,\\
\left(1+\frac{1}{\alpha-2}\right)N^{\alpha-1}&\text{ if }\alpha>2.
\end{array}
\right.
\end{align*}
\end{lemma}


\begin{proof}
$\bullet$ Let us start by proving \eqref{borne_H}. For $\textbf{x}\in\mathbb{R}^N$, we have
\begin{align*}
\mathcal{H}(\textbf{x})=&\sum_{i}|x_i|^2-\frac{1}{2N}\sum_{i\neq j}|x_i-x_j|\\
\geq&\sum_{i}|x_i|^2-\frac{1}{2N}\sum_{i\neq j}\left(2+\frac{1}{8}|x_i-x_j|^2\right),
\end{align*}
and thus
\begin{align*}
\mathcal{H}(\textbf{x})\geq&\sum_{i}|x_i|^2-\frac{N(N-1)}{N}-\frac{1}{8N}\sum_{i\neq j}|x_i|^2+|x_j|^2\\
\geq&\sum_{i}|x_i|^2\left(1-\frac{(N-1)}{4N}\right)-\frac{N(N-1)}{N}\\
\geq&\frac{1}{2}\sum_{i}|x_i|^2-N.
\end{align*}
Hence the result.\\

$\bullet$ We consider $(\textbf{X}_t)_t\geq0$ a solution of \eqref{particle_system} such that for all $t\geq0$ we have $X^1_t<...<X^N_t$. We apply Itô's formula to get, as almost surely  ${\forall t\geq0,\ \textbf{X}_t\in\mathcal{O}_N}$.
\begin{align*}
d\mathcal{H}(\textbf{X}_t)=&-\sum_{i}2\lambda\left|X^i_t\right|^2dt-2\sum_i\frac{X_i}{N}\sum_{j\neq i}V'(X^i_t-X^j_t)dt+2N\sigma_Ndt\\
&+2\sqrt{2\sigma_N}\sum_iX_idB^i_t+\sum_i\frac{\lambda X^i_t}{N}\sum_{j\neq i}\frac{X^i_t-X^j_t}{\left|X^i_t-X^j_t\right|}dt\\
&+\sum_i\left(\frac{1}{N}\sum_{j\neq i}\frac{X^i_t-X^j_t}{\left|X^i_t-X^j_t\right|}\right)\left(\frac{1}{N}\sum_{j\neq i}V'(X^i_t-X^j_t)\right)dt\\
&-\sqrt{2\sigma_N}\sum_i\left(\frac{1}{N}\sum_{j\neq i}\frac{X^i_t-X^j_t}{\left|X^i_t-X^j_t\right|}\right)dB^i_t.
\end{align*}
We have
\begin{align*}
\sum_iX_i\sum_{j\neq i}V'(X^i_t-X^j_t)=&\sum_{i>j}V'(X^i_t-X^j_t)(X^i_t-X^j_t)\\
=&-\sum_{i>j}\frac{1}{\left|X^i_t-X^j_t\right|^{\alpha-1}}\\
\sum_i\frac{\lambda X^i_t}{N}\sum_{j\neq i}\frac{X^i_t-X^j_t}{\left|X^i_t-X^j_t\right|}=&\frac{\lambda}{N}\sum_{i>j}\frac{\left(X^i_t-X^j_t\right)^2}{\left|X^i_t-X^j_t\right|}\\
=&\frac{\lambda}{N}\sum_{i>j}\left|X^i_t-X^j_t\right|=\frac{\lambda}{2N}\sum_{i\neq j}\left|X^i_t-X^j_t\right|.
\end{align*}
Hence
\begin{align*}
d\mathcal{H}(\textbf{X}_t)=&-2\lambda \mathcal{H}(\textbf{X}_t)dt-\frac{\lambda}{2N}\sum_{i\neq j}\left|X^i_t-X^j_t\right|+\frac{2}{N}\sum_{i>j}\frac{1}{\left|X^i_t-X^j_t\right|^{\alpha-1}}dt\\
&+2N\sigma_Ndt+\sum_i\left(\frac{1}{N}\sum_{j\neq i}\frac{X^i_t-X^j_t}{\left|X^i_t-X^j_t\right|}\right)\left(\frac{1}{N}\sum_{j\neq i}V'(X^i_t-X^j_t)\right)\\
&+2\sqrt{2\sigma_N}\sum_iX^i_tdB^i_t-\sqrt{2\sigma_N}\sum_i\left(\frac{1}{N}\sum_{j\neq i}\frac{X^i_t-X^j_t}{\left|X^i_t-X^j_t\right|}\right)dB^i_t.
\end{align*}
We now use the calculations of Lemma 4.2 of \cite{lu_mattingly} and write
\begin{align*}
\sum_i&\left(\frac{1}{N}\sum_{j\neq i}\frac{X^i_t-X^j_t}{|X^i_t-X^j_t|}\right)\left(\frac{1}{N}\sum_{j\neq i}V'(X^i_t-X^j_t)\right)\\
=&-\sum_i\left(\frac{1}{N}\sum_{j\neq i}\frac{X^i_t-X^j_t}{|X^i_t-X^j_t|}\right)\left(\frac{1}{N}\sum_{j\neq i}\frac{X^i_t-X^j_t}{|X^i_t-X^j_t|^{\alpha+1}}\right)\\
=&-\frac{1}{N^2}\sum_i\sum_{j,l\neq i}\frac{X^i_t-X^j_t}{|X^i_t-X^j_t|}\frac{X^i_t-X^l_t}{\left|X^i_t-X^l_t\right|^{\alpha+1}}\\
=&-\frac{1}{N^2}\sum_i\sum_{j\neq i}\frac{1}{|X^i_t-X^j_t|^\alpha}-\frac{1}{N^2}\sum_i\sum_{\scriptsize{\begin{array}{ll}j,l\neq i\\ j\neq l\end{array}}}\frac{X^i_t-X^j_t}{|X^i_t-X^j_t|}\frac{X^i_t-X^l_t}{\left|X^i_t-X^l_t\right|^{\alpha+1}},
\end{align*}
and
\begin{align*}
\sum_i&\sum_{j,l\neq i,j\neq l}\frac{X^i_t-X^j_t}{|X^i_t-X^j_t|}\frac{X^i_t-X^l_t}{|X^i_t-X^l_t|^{\alpha+1}}\\
=&\sum_i\sum_{\scriptsize{\begin{array}{ll}j,l\neq i\\ j< l\end{array}}}\frac{X^i_t-X^j_t}{|X^i_t-X^j_t|}\frac{X^i_t-X^l_t}{|X^i_t-X^l_t|^{\alpha+1}}+\frac{X^i_t-X^l_t}{|X^i_t-X^l_t|}\frac{X^i_t-X^j_t}{|X^i_t-X^j_t|^{\alpha+1}}\\
=&\sum_i\sum_{\scriptsize{\begin{array}{ll}j,l\neq i\\ j< l\end{array}}}\frac{(X^i_t-X^j_t)(X^i_t-X^l_t)}{|X^i_t-X^j_t||X^i_t-X^l_t|}\left(\frac{1}{|X^i_t-X^j_t|^{\alpha}}+\frac{1}{|X^i_t-X^l_t|^{\alpha}}\right)\\
=&\sum_{i<j< l}\frac{(X^i_t-X^j_t)(X^i_t-X^l_t)}{|X^i_t-X^j_t||X^i_t-X^l_t|}\left(\frac{1}{|X^i_t-X^j_t|^{\alpha}}+\frac{1}{|X^i_t-X^l_t|^{\alpha}}\right)\\
&\hspace{1cm}+\frac{(X^j_t-X^i_t)(X^j_t-X^l_t)}{|X^j_t-X^i_t||X^i_t-X^l_t|}\left(\frac{1}{|X^j_t-X^i_t|^{\alpha}}+\frac{1}{|X^j_t-X^l_t|^{\alpha}}\right)\\
&\hspace{1cm}+\frac{(X^l_t-X^j_t)(X^l_t-X^i_t)}{|X^l_t-X^j_t||X^l_t-X^i_t|}\left(\frac{1}{|X^l_t-X^j_t|^{\alpha}}+\frac{1}{|X^l_t-X^i_t|^{\alpha}}\right)\\
=&\sum_{i<j< l}\frac{1}{|X^i_t-X^j_t|^{\alpha}}+\frac{1}{|X^i_t-X^l_t|^{\alpha}}-\frac{1}{|X^j_t-X^i_t|^{\alpha}}-\frac{1}{|X^j_t-X^l_t|^{\alpha}}\\
&\hspace{1cm}+\frac{1}{|X^l_t-X^j_t|^{\alpha}}+\frac{1}{|X^l_t-X^i_t|^{\alpha}}\\
=&2\sum_{i<j< l}\frac{1}{|X^i_t-X^l_t|^{\alpha}}\\
=&2\sum_{i<j}\frac{j-i-1}{|X^i_t-X^j_t|^{\alpha}}.
\end{align*}
We therefore have
\begin{align*}
\sum_i\left(\frac{1}{N}\sum_{j\neq i}\frac{X^i_t-X^j_t}{|X^i_t-X^j_t|}\right)&\left(\frac{1}{N}\sum_{j\neq i}V'(X^i_t-X^j_t)\right)\\
=&-\frac{2}{N^2}\sum_{i>j}\frac{i-j}{|X^i_t-X^j_t|^{\alpha}}.
\end{align*}
We now compute
\begin{align}
d\left(e^{2\lambda t}\mathcal{H}(\textbf{X}_t)\right)=&2\lambda e^{2\lambda t}\mathcal{H}(\textbf{X}_t)dt+e^{2\lambda t}d\mathcal{H}(\textbf{X}_t)\nonumber\\
=&e^{2\lambda t}\left(\frac{2}{N}\sum_{i>j}\frac{1}{|X^i_t-X^j_t|^{\alpha-1}}-\frac{2}{N^2}\sum_{i>j}\frac{i-j}{|X^i_t-X^j_t|^{\alpha}}\right.\nonumber\\
&\hspace{1cm}\left.-\frac{\lambda}{2N}\sum_{i\neq j}\left|X^i_t-X^j_t\right|+2N\sigma_N\right)dt\nonumber\\
&+é\sqrt{2\sigma_N}e^{2\lambda t}\sum_i\left(2X^i_t-\frac{1}{N}\sum_{j\neq i}\frac{X^i_t-X^j_t}{|X^i_t-X^j_t|}\right)dB^i_t\label{ito_e_fois_H}
\end{align}

$\bullet$ Let $\alpha=1$. We get
\begin{align*}
&e^{2\lambda t}\mathbb{E}\mathcal{H}(\textbf{X}_t)\\
&\leq  \mathbb{E}\mathcal{H}(\textbf{X}_0)+(N-1+2N\sigma_N)\frac{e^{2\lambda t}-1}{2\lambda}-\mathbb{E}\left(\int_0^t\frac{2e^{2\lambda s}}{N^2}\sum_{j<i}\frac{i-j}{|X^i_s-X^j_s|}ds\right),
\end{align*}
hence \eqref{unif_tps_H} and \eqref{borne_interaction} for $\alpha=1$.

$\bullet$ Let $\alpha>1$. Using Young's inequality, we have, for all $\gamma>0$ and $i>j$
\begin{align*}
\frac{1}{|x|^{\alpha-1}}\leq\gamma^{\frac{\alpha}{\alpha-1}}\frac{\alpha-1}{\alpha}\frac{i-j}{|x|^{\alpha}}+\frac{1}{\alpha\gamma^\alpha(i-j)^{\alpha-1}}.
\end{align*}
Hence
\begin{align*}
\frac{1}{N}\sum_{i>j}\frac{1}{|X^i_t-X^j_t|^{\alpha-1}}\leq&\gamma^{\frac{\alpha}{\alpha-1}}\frac{\alpha-1}{\alpha}\frac{1}{N}\sum_{i>j}\frac{i-j}{|X^i_t-X^j_t|^{\alpha}}\\
&+\frac{1}{\alpha\gamma^\alpha}\frac{1}{N}\sum_{i>j}\frac{1}{(i-j)^{\alpha-1}}.
\end{align*}
We consider $\gamma^{\frac{\alpha}{\alpha-1}}=\frac{1}{N}$, i.e $\gamma^{\alpha}=\frac{1}{N^{\alpha-1}}$.
\begin{equation}\label{controle_serie}
\frac{1}{N}\sum_{i>j}\frac{1}{|X^i_t-X^j_t|^{\alpha-1}}\leq\frac{\alpha-1}{\alpha}\frac{1}{N^2}\sum_{i>j}\frac{i-j}{|X^i_t-X^j_t|^{\alpha}}+\frac{N^{\alpha-2}}{\alpha}\sum_{i>j}\frac{1}{(i-j)^{\alpha-1}}.\\
\end{equation}
Let us now assume $\alpha\in]1,2[$, using Lemma~\ref{comp_serie_int}
\begin{align*}
\sum_{i>j}\frac{1}{(i-j)^{\alpha-1}}&=\sum_{i=1}^N\sum_{j=1}^{i-1}\frac{1}{(i-j)^{\alpha-1}}=\sum_{i=1}^N\sum_{j=1}^{i-1}\frac{1}{j^{\alpha-1}}=\sum_{j=1}^N\frac{N-j}{j^{\alpha-1}}\\
&\leq N\sum_{j=1}^N\frac{1}{j^{\alpha-1}}\leq \frac{NN^{2-\alpha}}{2-\alpha}.
\end{align*}
Hence
\begin{align*}
\frac{1}{N}\sum_{i>j}\frac{1}{|X^i_t-X^j_t|^{\alpha-1}}\leq&\frac{\alpha-1}{\alpha}\frac{1}{N^2}\sum_{i>j}\frac{i-j}{|X^i_t-X^j_t|^{\alpha}}+\frac{N}{\alpha(2-\alpha)},
\end{align*}
and thus 
\begin{align*}
\frac{2}{N}\sum_{i>j}\frac{1}{|X^i_t-X^j_t|^{\alpha-1}}-\frac{2}{N^2}\sum_{i>j}\frac{i-j}{|X^i_t-X^j_t|^{\alpha}}\leq& -\frac{2}{\alpha}\frac{1}{N^2}\sum_{i>j}\frac{i-j}{|X^i_t-X^j_t|^{\alpha}}\\
&+\frac{2N}{\alpha(2-\alpha)}.
\end{align*}
Using \eqref{ito_e_fois_H}, we get
\begin{align}\label{gronwall_H}
e^{2\lambda t}\mathbb{E}\mathcal{H}(\textbf{X}_t)\leq&\mathbb{E}\mathcal{H}(\textbf{X}_0)-\frac{2}{\alpha}\mathbb{E}\left(\int_0^t\frac{e^{2\lambda s}}{N^2}\sum_{i>j}\frac{i-j}{|X^i_s-X^j_s|^{\alpha}}ds\right)\nonumber\\
&+\int_0^te^{2\lambda s}\left(2N\sigma_N+\frac{2N}{\alpha(2-\alpha)}\right)ds.
\end{align}
This yields
\begin{align*}
\frac{2}{\alpha}&\mathbb{E}\left(\int_0^t\frac{e^{2\lambda s}}{N^2}\sum_{i>j}\frac{i-j}{|X^i_s-X^j_s|^{\alpha}}ds\right)\\
&\hspace{0.3cm}\leq\mathbb{E}\mathcal{H}(\textbf{X}_0)-e^{2\lambda t}\mathbb{E}\mathcal{H}(\textbf{X}_t)+\int_0^te^{2\lambda s}\left(2N\sigma_N+\frac{2N}{\alpha(2-\alpha)}\right)ds\\
&\hspace{0.3cm}\leq\mathbb{E}\mathcal{H}(\textbf{X}_0)+Ne^{2\lambda t}+\left(2N\sigma_N+\frac{2N}{\alpha(2-\alpha)}\right)\left(\frac{e^{2\lambda t}-1}{2\lambda}\right).
\end{align*}
This yields the desired result for $\alpha\in]1,2[$.\\

Let $\alpha=2$. Instead of the control \eqref{controle_serie}, we have, by Lemma~\ref{comp_serie_int}
\begin{align*}
\sum_{i>j}\frac{1}{(i-j)^{\alpha-1}}\leq 2N\ln N,
\end{align*}
which then yields
\begin{align*}
2\mathbb{E}&\left(\int_0^t\frac{e^{2\lambda s}}{N^2}\sum_{i>j}\frac{i-j}{\left|X^i_s-X^j_s\right|^{\alpha}}ds\right)
\\
\leq&\alpha\left(\mathbb{E}\mathcal{H}(\textbf{X}_0)+Ne^{2\lambda t}+2N\sigma_N \frac{e^{2\lambda t}-1}{2\lambda}\right)+4N\ln N\frac{e^{2\lambda t}-1}{2\lambda}.\\
\end{align*}
Finally, let $\alpha>2$. By Lemma~\ref{comp_serie_int}
\begin{align*}
\sum_{i>j}\frac{1}{(i-j)^{\alpha-1}}\leq \left(1+\frac{1}{\alpha-2}\right)N,
\end{align*}
which then yields
\begin{align*}
2\mathbb{E}&\left(\int_0^t\frac{e^{2\lambda s}}{N^2}\sum_{i>j}\frac{i-j}{\left|X^i_s-X^j_s\right|^{\alpha}}ds\right)
\\
&\hspace{1cm}\leq\alpha\left(\mathbb{E}\mathcal{H}(\textbf{X}_0)+Ne^{2\lambda t}+2N\sigma_N \frac{e^{2\lambda t}-1}{2\lambda}\right)\\
&\hspace{1.3cm}+2\left(1+\frac{1}{\alpha-2}\right)N^{\alpha-1}\frac{e^{2\lambda t}-1}{2\lambda}.
\end{align*}

$\bullet$ Using \eqref{gronwall_H} for $\alpha>1$
\begin{align*}
e^{2\lambda t}\mathbb{E}\mathcal{H}(\textbf{X}_t)\leq\mathbb{E}\mathcal{H}(\textbf{X}_0)+\left(2N\sigma_N+\frac{2C(\alpha,N)}{\alpha}\right)\frac{e^{2\lambda t}-1}{2\lambda },
\end{align*}
i.e
\begin{align*}
\mathbb{E}\mathcal{H}(\textbf{X}_t)\leq e^{-2\lambda t}\mathbb{E}\mathcal{H}(\textbf{X}_0)+\frac{1}{\lambda }\left(N\sigma_N+\frac{C(\alpha,N)}{\alpha}\right).
\end{align*}
Hence the uniform in time bound.

$\bullet$ We now wish to prove \eqref{borne_interaction_sup}. Using the previous calculations, we have
\begin{align*}
d\mathcal{H}(\textbf{X}_t)\leq&-2\lambda \mathcal{H}(\textbf{X}_t)dt+2\sum_{i>j}\frac{1}{\left|X^i_t-X^j_t\right|^{\alpha-1}}dt+2N\sigma_Ndt\\
&-\frac{2}{N^2}\sum_{i<j}\frac{j-i}{\left|X^i_t-X^j_t\right|^\alpha}dt\nonumber\\
&+2\sqrt{2\sigma_N}\sum_iX_idB^i_t-\sqrt{2\sigma_N}\sum_i\left(\frac{1}{N}\sum_{j\neq i}\frac{X^i_t-X^j_t}{\left|X^i_t-X^j_t\right|}\right)dB^i_t,
\end{align*}
as well as 
\begin{align}
d\mathcal{H}(&\textbf{X}_t)\nonumber\\
\leq&-2\lambda \mathcal{H}(\textbf{X}_t)dt+2N\sigma_Ndt+\frac{2C(\alpha,N)}{\alpha}dt-\frac{2}{\alpha N^2}\sum_{i>j}\frac{i-j}{|X^i_t-X^j_t|^{\alpha}}dt\nonumber\\
&+2\sqrt{2\sigma_N}\sum_iX_idB^i_t-\sqrt{2\sigma_N}\sum_i\left(\frac{1}{N}\sum_{j\neq i}\frac{X^i_t-X^j_t}{\left|X^i_t-X^j_t\right|}\right)dB^i_t\label{estimee_sup_2}.
\end{align}
Hence, from \eqref{estimee_sup_2}, we get
\begin{align*}
\frac{2}{\alpha}\mathbb{E}\left(\frac{1}{N^2}\int_0^t\sum_{i>j}\frac{i-j}{|X^i_s-X^j_s|^{\alpha}}ds\right)\leq& \mathbb{E}\mathcal{H}(\textbf{X}_0)-\mathbb{E}\mathcal{H}(\textbf{X}_t)-\int_0^t2\lambda\mathbb{E}\mathcal{H}(\textbf{X}_s)ds\\
&+\left(2N\sigma_N+\frac{2C(\alpha,N)}{\alpha}\right)t\\
\leq& \mathbb{E}\mathcal{H}(\textbf{X}_0)+N\\
&+\left(2\lambda N+2N\sigma_N+\frac{2C(\alpha,N)}{\alpha}\right)t.
\end{align*}
\end{proof}

%
%
%
%

\section{Limit for large number of particles with vanishing noise}\label{sec_prop}

Consider for a given $N\geq1$ a solution $X_t=(X^1_t, ..., X^N_t)$ of \eqref{particle_system}. Our goal is to prove the following theorem

\begin{theorem}\label{thm_resume_prop}
Consider a sequence of initial empirical measures $(\mu^N_0)_{N\geq1}$ such that there exists $\bar{\rho}_0 \in \mathcal P_2(\mathbb R)$ such that $\lim_{N\rightarrow0}\mathbb{E}\left(\mathcal{W}_2(\mu^N_0,\bar{\rho}_0)^2\right)=0$. Under Assumptions~\ref{Hyp_U_conv}~and~\ref{Hyp_V}, for $\alpha\in[1,2[$ (with the additional assumption $\sigma_N\leq\frac{1}{N}$ for $\alpha=1$), there exist a deterministic family of measures $(\rho_t)_{t\geq 0}\in\mathcal{C}(\mathbb{R}^+,\mathcal{P}_2(\mathbb{R}))$, as well as universal constants $C_1,C_2>0$ and a quantity $C_0^N>0$ that depends on the initial condition and such that $C_0^N\rightarrow0$ as $N\rightarrow\infty$, such that for all $N\geq1$ and all $t\geq0$
\begin{align*}
\mathbb{E}\left(\mathcal{W}_2(\mu^N_t,\bar{\rho}_t)^2\right)\leq e^{-2\lambda t}C_0^N+\frac{C_1}{N^{(2-\alpha)/\alpha}}+C_2\sigma_N.
\end{align*}
\end{theorem}

%
%

\subsection{... with $\alpha=1$.}\label{sec_prop_alpha_1}

Let us start with the case $\alpha=1$, as this will allow us to describe the method in an easier case, before extending the result to $\alpha\in]1,2[$ using much more cumbersome computations.


\begin{lemma}\label{cauchy_alpha_1}
Consider Assumption~\ref{Hyp_U_conv} and Assumption~\ref{Hyp_V}, with $\alpha=1$  and $\sigma_N\leq\frac{1}{N}$. Let $(\mu^N)_{N\in\mathbb{N}}$ be any sequence of independent empirical measures, such that $\mu_t^N$ is the empirical measure of the $N$ particle system at time $t$. We have for all $t\geq0$ and all $N,M\geq 1$
\begin{equation}\label{unif_sans_sup}
\mathbb{E}\left(\mathcal{W}_2\left(\mu_t^N,\mu_t^M\right)^2\right)\leq e^{-2\lambda t}\mathbb{E}\left(\mathcal{W}_2\left(\mu_0^N,\mu_0^M\right)^2\right)+\frac{1}{2\lambda}\left(\frac{1}{N}+\frac{1}{M}+2\left(\sigma_N+\sigma_M\right)\right),
\end{equation}
 There also are constants $C_1, C_2, C_3>0$ independent of $N$ and $M$ such that
\begin{align}
\mathbb{E}\left(\sup_{s\in[0,t]}\mathcal{W}_2\left(\mu_s^N,\mu_s^M\right)^2\right)\leq& e^{C_1t}\left(\mathbb{E}\left(\mathcal{W}_2\left(\mu_0^N,\mu_0^M\right)^2\right)+C_2(\sigma_M+\sigma_N)\right.\nonumber\\
&\hspace{1cm}\left.+C_3\left(\frac{1}{M}+\frac{1}{N}\right)\right).\label{non_unif_avec_sup}
\end{align}
\end{lemma}


\begin{proof}
For $N,M\geq1$, let $(\tilde{B}^i)_{i\in\{1,...,M\}}$ and $(\tilde{B}^{'j})_{j\in\{1,...,N\}}$ be two independent families of Brownian motions, and consider $x_1<...<x_M$ and $y_1<...<y_N$ two sets of initial conditions. Denote by $(\tilde{X}^{i,M})_{i\in\{1,...,M\}}$ (resp. $(\tilde{Y}^{j,N})_{j\in\{1,...,N\}}$) the unique strong solution of \eqref{particle_system} with initial conditions $x_1<...<x_M$ and Brownian motions $(\tilde{B}^i)_{i\in\{1,...,M\}}$ (resp. initial conditions $y_1<...<y_N$ and Brownian motions $(\tilde{B'}^j)_{j\in\{1,...,N\}}$).

In order to compare the two sets $(\tilde{X}^{i,M})_{i\in\{1,...,M\}}$ and $(\tilde{Y}^{j,N})_{j\in\{1,...,M\}}$ despite the difference in the number of particles, we consider $N$ exact copies of the system $(\tilde{X}^{i,M})_{i\in\{1,...,M\}}$, and $M$ exact copies of $(\tilde{Y}^{j,N})_{j\in\{1,...,N\}}$. We denote $(X^{i})_{i\in\{1,...,NM\}}$ and $(Y^{i})_{i\in\{1,...,NM\}}$ the resulting processes, numbered such that for all $t\geq0$
\begin{align*}
-\infty<&X^1_t=...=X^N_t<...<X^{N(M-1)+1}_t=...=X^{NM}_t<\infty\\
-\infty<&Y^1_t=...=Y^M_t<...<Y^{M(N-1)+1}_t=...=Y^{NM}_t<\infty.
\end{align*}
Thus
\begin{align*}
\mu^M_t=\frac{1}{M}\sum_{i=1}^M\delta_{\tilde{X}^{i,M}_t}=\frac{1}{NM}\sum_{i=1}^{NM}\delta_{X^{i}_t},\ \ \ \ \mu^N_t=\frac{1}{N}\sum_{i=1}^N\delta_{\tilde{Y}^{i,N}_t}=\frac{1}{NM}\sum_{i=1}^{NM}\delta_{Y^{i}_t},
\end{align*}
and
\begin{align*}
\mathcal{W}_2\left(\mu_t^N,\mu_t^M\right)^2=\frac{1}{NM}\sum_{i=1}^{NM}\left|X^i_t-Y^i_t\right|^2.
\end{align*}
By convention, and for the sake of clarity, consider $V'(0)=0$. Then, for all $i\in\{1,...,NM\}$, we have the following dynamics
\begin{align*}
dX^i_t=&-\lambda X^i_tdt-\frac{1}{NM}\sum_{j}V'(X^i_t-X^j_t)dt+\sqrt{2\sigma_M}dB^{i}_t\\
dY^i_t=&-\lambda Y^i_tdt-\frac{1}{NM}\sum_{j}V'(Y^i_t-Y^j_t)dt+\sqrt{2\sigma_N}dB^{'i}_t,
\end{align*}
where the Brownian motions $(B^{i}_t)_i$ and (resp. $(B^{'i}_t)_i$) are such that for all $k\in\{1,...,M\}$, we have ${B^{N(k-1)+1}=...=B^{Nk}=\tilde{B}^{k}}$, (resp. for all $l\in\{1,...,N\}$, ${B'^{M(l-1)+1}=...=B'^{Ml}=\tilde{B}'^{l}}$).
Thus
\begin{align*}
d\left(X^i_t-Y^i_t\right)^2=&-2\lambda \left(X^i_t-Y^i_t\right)^2dt+2\sigma_Mdt+2\sigma_Ndt\\
&-2\left(X^i_t-Y^i_t\right)\frac{1}{NM}\sum_{j}\left(V'(X^i_t-X^j_t)-V'(Y^i_t-Y^j_t)\right)dt\\
&+2\sqrt{2\sigma_M}\left(X^i_t-Y^i_t\right)dB^{i}_t-2\sqrt{2\sigma_N}\left(X^i_t-Y^i_t\right)dB^{'i}_t,
\end{align*}
and
\begin{align*}
d&\left(\frac{1}{NM}\sum_i\left(X^i_t-Y^i_t\right)^2\right)\\
&\hspace{1cm}=-\frac{2\lambda}{NM}\sum_i\left(X^i_t-Y^i_t\right)^2dt+2(\sigma_M+\sigma_N)dt\\
&\hspace{1.2cm}+\frac{2\sqrt{2\sigma_M}}{NM}\sum_i\left(X^i_t-Y^i_t\right)dB^{i}_t-\frac{2\sqrt{2\sigma_N}}{NM}\sum_i\left(X^i_t-Y^i_t\right)dB^{'i}_t\\
&\hspace{1.2cm}-\frac{2}{(NM)^2}\sum_i\left(X^i_t-Y^i_t\right)\sum_j\left(V'(X^i_t-X^j_t)-V'(Y^i_t-Y^j_t)\right)dt.
\end{align*}
We first compute
\begin{align*}
\sum_i&\left(X^i_t-Y^i_t\right)\sum_j\left(V'(X^i_t-X^j_t)-V'(Y^i_t-Y^j_t)\right)\\
=&\sum_{i>j}\left(V'(X^i_t-X^j_t)-V'(Y^i_t-Y^j_t)\right)\left(\left(X^i_t-Y^i_t\right)-\left(X^j_t-Y^j_t\right)\right)\\
=&\sum_{i>j}\left(V'(X^i_t-X^j_t)-V'(Y^i_t-Y^j_t)\right)\left(\left(X^i_t-X^j_t\right)-\left(Y^i_t-Y^j_t\right)\right).
\end{align*}
Remember that the function $x\rightarrow V'(x)$ is increasing for $x>0$. Thus, all choices of indexes $i>j$ such that $X^i_t\neq X^j_t$ (which therefore imply, by the choice of numbering, that $X^i_t> X^j_t$) and $Y^i_t\neq Y^j_t$ yield nonnegative terms in the sum above. If $X^i_t= X^j_t$, by convention, we have $V'(X^i_t- X^j_t)=0$.
\begin{align}
\sum_i\left(X^i_t-Y^i_t\right)&\sum_j\left(V'(X^i_t-X^j_t)-V'(Y^i_t-Y^j_t)\right)\nonumber\\
\geq&\sum_{i>j\text{ s.t }Y^i_t= Y^j_t }V'(X^i_t-X^j_t)\left(X^i_t-X^j_t\right)\nonumber\\
&+\sum_{i>j\text{ s.t }X^i_t= X^j_t }V'(Y^i_t-Y^j_t)\left(Y^i_t-Y^j_t\right)\label{problem}\\
\geq&\sum_{i>j\text{ s.t }Y^i_t= Y^j_t }-1+\sum_{i>j\text{ s.t }X^i_t= X^j_t }-1\nonumber\\
=&-\frac{M(M-1)}{2}N-\frac{N(N-1)}{2}M\nonumber.
\end{align}
We thus obtain, for all $t\geq0$
\begin{align}
\mathcal{W}_2&\left(\mu_t^N,\mu_t^M\right)^2\nonumber\\
\leq&\mathcal{W}_2\left(\mu_0^N,\mu_0^M\right)^2-2\lambda\int_0^t\mathcal{W}_2\left(\mu_s^N,\mu_s^M\right)^2ds\nonumber\\
&+\int_0^t\left(\frac{2}{(NM)^2}\left(\frac{N(N-1)}{2}M+\frac{M(M-1)}{2}N\right)+2(\sigma_M+\sigma_N)\right)ds\nonumber\\
&+\frac{2\sqrt{2\sigma_M}}{NM}\sum_i\int_0^t\left(X^i_s-Y^i_s\right)dB^{i}_s-\frac{2\sqrt{2\sigma_N}}{NM}\sum_i\int_0^t\left(X^i_s-Y^i_s\right)dB^{'i}_s\label{esti_gron_alpha_1_sans_sup}.
\end{align}
Considering the expectation of the inequality above, and using Gronwall's lemma yields \eqref{unif_sans_sup}. Let us now take the supremum
\begin{align*}
\mathbb{E}\left(\sup_{s\in[0,t]}\mathcal{W}_2\left(\mu_s^N,\mu_s^M\right)^2\right)\leq&\mathbb{E}\left(\mathcal{W}_2\left(\mu_0^N,\mu_0^M\right)^2\right)+\left(\frac{1}{M}+\frac{1}{N}\right)t+2(\sigma_M+\sigma_N)t\\
&+\mathbb{E}\left(\frac{2\sqrt{2\sigma_M}}{NM}\sup_{s\in[0,t]}\sum_i\int_0^s\left(X^i_u-Y^i_u\right)dB^{i}_u\right)\\
&+\mathbb{E}\left(\frac{2\sqrt{2\sigma_N}}{NM}\sup_{s\in[0,t]}\sum_i-\int_0^s\left(X^i_u-Y^i_u\right)dB^{'i}_u\right).
\end{align*}
We use Burkholder-Davis-Gundy inequality to show that there exists a constant $C_{BDG}$ such that
\begin{align*}
\mathbb{E}\left(\frac{2\sqrt{2\sigma_M}}{NM}\right.&\left.\sup_{s\in[0,t]}\sum_i\int_0^s\left(X^i_u-Y^i_u\right)dB^{i}_u\right)\\
\leq&\frac{2\sqrt{2\sigma_M}}{NM}\sum_i\mathbb{E}\left(\sup_{s\in[0,t]}\int_0^s\left(X^i_u-Y^i_u\right)dB^{i}_u\right)\\
\leq& C_{BDG}\frac{2\sqrt{2\sigma_M}}{NM}\sum_i\mathbb{E}\left(\left(\int_0^t\left(X^i_s-Y^i_s\right)^2ds\right)^{1/2}\right)\\
\leq& C_{BDG}\frac{2\sqrt{2\sigma_M}}{NM}\sum_i\left(\frac{1}{2\sqrt{2\sigma_M}}\mathbb{E}\left(\int_0^t\left(X^i_s-Y^i_s\right)^2ds\right)+\frac{\sqrt{2\sigma_M}}{2}\right)\\
=&C_{BDG}\mathbb{E}\left(\int_0^t\frac{1}{NM}\sum_i\left(X^i_s-Y^i_s\right)^2ds\right)+2\sigma_MC_{BDG}\\
=&C_{BDG}\mathbb{E}\left(\int_0^t\mathcal{W}_2\left(\mu_s^N,\mu_s^M\right)^2ds\right)+2\sigma_MC_{BDG}.
\end{align*}
Using the same control on the second local martingale, we get
\begin{align*}
\mathbb{E}&\left(\sup_{s\in[0,t]}\mathcal{W}_2\left(\mu_s^N,\mu_s^M\right)^2\right)\\
&\hspace{1cm}\leq\mathbb{E}\left(\mathcal{W}_2\left(\mu_0^N,\mu_0^M\right)^2\right)+\left(\frac{1}{N}+\frac{1}{M}\right)t+2(\sigma_M+\sigma_N)t\\
&\hspace{1.3cm}+2C_{BDG}(\sigma_M+\sigma_N)+2C_{BDG}\int_0^t\mathbb{E}\left(\mathcal{W}_2\left(\mu_s^N,\mu_s^M\right)^2\right)ds,
\end{align*}
and thus, denoting $$C_{N,M}=\frac{1}{2C_{BDG}}\left(\frac{1}{M}+\frac{1}{N}+2(\sigma_M+\sigma_N)\right),$$ we get
\begin{align*}
\mathbb{E}\left(\sup_{s\in[0,t]}\mathcal{W}_2\left(\mu_s^N,\mu_s^M\right)^2\right)&+C_{N,M}\\
\leq&\mathbb{E}\left(\mathcal{W}_2\left(\mu_0^N,\mu_0^M\right)^2\right)+2C_{BDG}(\sigma_M+\sigma_N)+C_{N,M}\\
&+2C_{BDG}\int_0^t\left(\mathbb{E}\left(\mathcal{W}_2\left(\mu_s^N,\mu_s^M\right)^2\right)+C_{N,M}\right)ds.
\end{align*}
Gronwall's lemma yields \eqref{non_unif_avec_sup}.
\end{proof}

\begin{remark}
For $\lambda=0$, the proof above still yields a quantitative result of propagation of chaos, though no longer uniform in time : considering \eqref{esti_gron_alpha_1_sans_sup}, we get for all $t\geq0$
\begin{equation*}
\mathbb{E}\mathcal{W}_2\left(\mu_t^N,\mu_t^M\right)^2\leq\mathbb{E}\mathcal{W}_2\left(\mu_0^N,\mu_0^M\right)^2+\left(\frac{1}{N}+\frac{1}{M}+2(\sigma_N+\sigma_M)\right)t.
\end{equation*}
Likewise, under Assumption~\ref{Hyp_U_lip} instead of Assumption~\ref{Hyp_U_conv}, we get a similar (non uniform in time) result
\begin{equation*}
\mathbb{E}\mathcal{W}_2\left(\mu_t^N,\mu_t^M\right)^2\leq e^{2L_Ut}\left(\mathbb{E}\mathcal{W}_2\left(\mu_0^N,\mu_0^M\right)^2+\frac{1}{2L_U}\left(\frac{1}{N}+\frac{1}{M}+2(\sigma_N+\sigma_M)\right)\right).
\end{equation*}
\end{remark}

%
%

\subsection{... with $\alpha\in]1,2[$.}\label{sec_prop_alpha_sup_1}

Let us now show the proof of Lemma~\ref{cauchy_alpha_1} can be extended to other values of $\alpha$. Notice that we use the assumption $\alpha=1$ to deal with \eqref{problem}. 
To account for this quantity for $\alpha>1$, we now use the bound \eqref{borne_interaction} and obtain, using the definition of $\mathcal{H}$ given in \eqref{def_H_cal}, the following lemma.


\begin{lemma}
Consider Assumptions~\ref{Hyp_U_conv}~and~\ref{Hyp_V}, with $\alpha\in]1,2[$. Let $(\mu^N)_{N\in\mathbb{N}}$ be any sequence of independent empirical measures, such that $\mu_t^N$ is the empirical measure of the $N$ particle system at time $t$. We have for all $t\geq0$ and all $N,M\geq 1$
\begin{align}
\mathbb{E}&\left(\mathcal{W}_2\left(\mu_t^N,\mu_t^M\right)^2\right)
\label{unif_sans_sup_2}\\
\leq&e^{-2\lambda t}\left(\mathbb{E}\left(\mathcal{W}_2\left(\mu_0^N,\mu_0^M\right)^2\right)+\frac{3(\alpha-1)}{N^{(2-\alpha)/\alpha}}\mathbb{E}^{\mu_0^M}\left(|X|^2\right)+\frac{3(\alpha-1)}{M^{(2-\alpha)/\alpha}}\mathbb{E}^{\mu_0^N}\left(|Y|^2\right)\right)\nonumber\\
&+\frac{1}{\lambda}(\sigma_M+\sigma_N)+\frac{1}{\lambda}\left(\frac{3\alpha-2}{\alpha(2-\alpha)}+3\lambda(\alpha-1)\right)\left(\frac{1}{N^{(2-\alpha)/\alpha}}+\frac{1}{M^{(2-\alpha)/\alpha}}\right)\nonumber\\
&+\frac{3(\alpha-1)}{2\lambda}\left(\frac{\sigma_N}{M^{(2-\alpha)/\alpha}}+\frac{\sigma_M}{N^{(2-\alpha)/\alpha}}\right)\nonumber
\end{align}
\end{lemma}


\begin{proof}
Consider a similar set up as the proof of Lemma~\ref{cauchy_alpha_1}, and define $(\tilde{X}^{i,M}_t)_i$,$(\tilde{Y}^{j,N}_t)_j$, $(X^{i}_t)_i$, $(Y^{j}_t)_j$ in the same manner. We compute, like previously
\begin{align*}
d\left(e^{2\lambda t}\mathcal{W}_2\left(\mu_t^N,\mu_t^M\right)^2\right)=&2\lambda e^{2\lambda t}\mathcal{W}_2\left(\mu_t^N,\mu_t^M\right)^2dt+e^{2\lambda t}d\mathcal{W}_2\left(\mu_t^N,\mu_t^M\right)^2\\
=&e^{2\lambda t}A_tdt+e^{2\lambda t}dM_t
\end{align*}
where $M_t$ is a local martingale, and
\begin{align*}
A_t\leq& -\frac{2}{(NM)^2}\sum_{i>j\text{ s.t }Y^i_t= Y^j_t }V'(X^i_t-X^j_t)\left(X^i_t-X^j_t\right)\\
&-\frac{2}{(NM)^2}\sum_{i>j\text{ s.t }X^i_t= X^j_t }V'(Y^i_t-Y^j_t)\left(Y^i_t-Y^j_t\right)+2(\sigma_M+\sigma_N).
\end{align*}
Using Young's inequality, we have, for all $\gamma>0$ and $i>j$
\begin{align*}
\frac{1}{|x|^{\alpha-1}}\leq\gamma^{\frac{\alpha}{\alpha-1}}\frac{\alpha-1}{\alpha}\frac{\lfloor\frac{i-j}{N}\rfloor+1}{|x|^{\alpha}}+\frac{1}{\alpha\gamma^\alpha\left(\lfloor\frac{i-j}{N}\rfloor+1\right)^{\alpha-1}}.
\end{align*}
Hence
\begin{align*}
\frac{1}{(NM)^2}\sum_{i>j\text{ s.t }\scriptsize{\begin{array}{ll}Y^i_t= Y^j_t\\X^i_t\neq X^j_t\end{array}} }&\frac{1}{|X^i_t-X^j_t|^{\alpha-1}}\\
\leq&\frac{1}{(NM)^2}\gamma^{\frac{\alpha}{\alpha-1}}\frac{\alpha-1}{\alpha}\sum_{i>j\text{ s.t }\scriptsize{\begin{array}{ll}Y^i_t= Y^j_t\\X^i_t\neq X^j_t\end{array}} }\frac{\lfloor\frac{i-j}{N}\rfloor+1}{|X^i_t-X^j_t|^{\alpha}}\\
&+\frac{1}{(NM)^2}\frac{1}{\alpha\gamma^\alpha}\sum_{i>j\text{ s.t }Y^i_t= Y^j_t }\frac{1}{\left(\lfloor\frac{i-j}{N}\rfloor+1\right)^{\alpha-1}}.
\end{align*}
We calculate, since $\lfloor\frac{i-j}{N}\rfloor+1\geq\frac{i-j}{N}$
\begin{align*}
&\sum_{i>j\text{ s.t }Y^i_t= Y^j_t }\frac{1}{\left(\lfloor\frac{i-j}{N}\rfloor+1\right)^{\alpha-1}}\leq\sum_{i>j\text{ s.t }Y^i_t= Y^j_t }\frac{N^{\alpha-1}}{\left( i-j\right)^{\alpha-1}}\ \ \ \text{ and }\\
&\sum_{i>j\text{ s.t }Y^i_t= Y^j_t }\frac{1}{(i-j)^{\alpha-1}}=\sum_{i=1}^{NM}\sum_{j=\lfloor\frac{i-1}{M}\rfloor M+1}^{i-1}\frac{1}{(i-j)^{\alpha-1}}=\sum_{i=1}^{NM}\sum_{j=1}^{i-1-\left\lfloor\frac{i-1}{M}\right\rfloor M}\frac{1}{j^{\alpha-1}}
\end{align*}
which implies 
\begin{align*}
\sum_{i>j\text{ s.t }Y^i_t= Y^j_t }\frac{1}{\left(\lfloor\frac{i-j}{N}\rfloor+1\right)^{\alpha-1}}\leq&N^{\alpha-1}\sum_{i=1}^{NM}\frac{1}{2-\alpha}\left(i-1-\left\lfloor\frac{i-1}{M}\right\rfloor M\right)^{2-\alpha}\\
\leq&\frac{N^\alpha MM^{2-\alpha}}{2-\alpha}.
\end{align*}
Hence
\begin{align*}
\frac{1}{(NM)^2}\frac{1}{\alpha\gamma^\alpha}\sum_{i>j\text{ s.t }Y^i_t= Y^j_t }\frac{1}{\left(\lfloor\frac{i-j}{N}\rfloor+1\right)^{\alpha-1}}\leq\frac{1}{N^{2-\alpha}}\frac{1}{M^{\alpha-1}}\frac{1}{\alpha(2-\alpha)\gamma^\alpha}.
\end{align*}
We consider $\gamma^{\frac{\alpha}{\alpha-1}}=\frac{1}{M^{1+\delta}}$ for some yet unspecified $\delta>0$. Thus
\begin{align*}
\frac{1}{(NM)^2}\frac{1}{\alpha\gamma^\alpha}&\sum_{i>j\text{ s.t }Y^i_t= Y^j_t }\frac{1}{\left(\lfloor\frac{i-j}{N}\rfloor+1\right)^{\alpha-1}}\\
&\leq\frac{1}{\alpha(2-\alpha) N^{2-\alpha}}\frac{M^{(1+\delta)(\alpha-1)}}{M^{\alpha-1}}=\frac{M^{\delta(\alpha-1)}}{\alpha(2-\alpha) N^{2-\alpha}}.
\end{align*}
Furthermore
\begin{align*}
\frac{1}{(NM)^2}\sum_{i>j\text{ s.t }\scriptsize{\begin{array}{ll}Y^i_t= Y^j_t\\X^i_t\neq X^j_t\end{array}} }&\frac{\lfloor\frac{i-j}{N}\rfloor+1}{|X^i_t-X^j_t|^{\alpha}}\\
\leq& \frac{1}{(NM)^2}\sum_{i=0}^{M-1}\sum_{j=0}^{i-1}\sum_{k=1}^N\sum_{l=1}^N\frac{\lfloor i-j+\frac{k-l}{N}\rfloor+1}{|X^{iN+k}_t-X^{jN+l}_t|^{\alpha}}\\
\leq&\frac{1}{(NM)^2}\sum_{i=1}^{M}\sum_{j=1}^{i-1}N^2\frac{i-j+2}{|\tilde{X}^{i}_t-\tilde{X}^{j}_t|^{\alpha}},
\end{align*}
and thus
\begin{align*}
\gamma^{\frac{\alpha}{\alpha-1}}\frac{\alpha-1}{\alpha}\frac{1}{(NM)^2}\sum_{i>j\text{ s.t }\scriptsize{\begin{array}{ll}Y^i_t= Y^j_t\\X^i_t\neq X^j_t\end{array}} }\frac{\lfloor\frac{i-j}{N}\rfloor+1}{|X^i_t-X^j_t|^{\alpha}}\leq&3\frac{\alpha-1}{\alpha}\frac{1}{M^{1+\delta}}\frac{1}{M^2}\sum_{i>j}\frac{i-j}{|\tilde{X}^{i}_t-\tilde{X}^{j}_t|^{\alpha}}.
\end{align*}
Using the same calculations to deal with 
\[
\frac{1}{(NM)^2}\sum_{i>j\text{ s.t }\scriptsize{\begin{array}{ll}X^i_t= X^j_t\\Y^i_t\neq Y^j_t\end{array}} }\frac{1}{|Y^i_t-Y^j_t|^{\alpha-1}},\]
 we obtain by taking the expectation in Itô's formula that for all $t\geq0$
\begin{align*}
e^{2\lambda t}&\mathbb{E}\left(\mathcal{W}_2\left(\mu_t^N,\mu_t^M\right)^2\right)\\
\leq&\mathbb{E}\left(\mathcal{W}_2\left(\mu_0^N,\mu_0^M\right)^2\right)+2(\sigma_M+\sigma_N)\int_0^te^{2\lambda s}ds\\
&+2\int_0^t\mathbb{E}\left(\frac{e^{2\lambda s}}{(NM)^2}\sum_{i>j\text{ s.t }Y^i_s= Y^j_s }\frac{1}{|X^i_s-X^j_s|^{\alpha-1}}\right)ds\\
&+2\int_0^t\mathbb{E}\left(\frac{e^{2\lambda s}}{(NM)^2}\sum_{i>j\text{ s.t }X^i_s= X^j_s }\frac{1}{|Y^i_s-Y^j_s|^{\alpha-1}}\right)ds\\
\leq&\mathbb{E}\left(\mathcal{W}_2\left(\mu_0^N,\mu_0^M\right)^2\right)+2(\sigma_M+\sigma_N)\int_0^te^{2\lambda s}ds+\int_0^t\frac{2M^{\delta(\alpha-1)}}{\alpha(2-\alpha) N^{2-\alpha}}e^{2\lambda s}ds\\
&+\int_0^t\frac{2N^{\tilde{\delta}(\alpha-1)}}{\alpha(2-\alpha) M^{2-\alpha}}e^{2\lambda s}ds\\
&+6\frac{\alpha-1}{\alpha}\frac{1}{M^{1+\delta}}\int_0^t\mathbb{E}\left(\frac{e^{2\lambda s}}{M^2}\sum_{i>j}\frac{i-j}{|\tilde{X}^{i,M}_s-\tilde{X}^{j,M}_s|^{\alpha}}\right)ds\\
&+6\frac{\alpha-1}{\alpha}\frac{1}{N^{1+\tilde{\delta}}}\int_0^t\mathbb{E}\left(\frac{e^{2\lambda s}}{N^2}\sum_{i>j}\frac{i-j}{|\tilde{Y}^{i,N}_s-\tilde{Y}^{j,N}_s|^{\alpha}}\right)ds,
\end{align*}
and we use \eqref{borne_interaction} to get
\begin{align*}
\mathbb{E}&\left(e^{2\lambda t}\mathcal{W}_2\left(\mu_t^N,\mu_t^M\right)^2\right)
\\
\leq&\mathbb{E}\left(\mathcal{W}_2\left(\mu_0^N,\mu_0^M\right)^2\right)+2(\sigma_M+\sigma_N)\int_0^te^{2\lambda s}ds\\
&+\int_0^t\frac{2M^{\delta(\alpha-1)}}{\alpha(2-\alpha) N^{2-\alpha}}e^{2\lambda s}ds+\int_0^t\frac{2N^{\tilde{\delta}(\alpha-1)}}{\alpha(2-\alpha) M^{2-\alpha}}e^{2\lambda s}ds\\
&+\frac{3(\alpha-1)}{M^{1+\delta}}\left(\mathbb{E}\left(\mathcal{H}((\tilde{X}^{i,M}_0)_{i})\right)+Me^{2\lambda t}+2M\sigma_M\frac{e^{2\lambda t-1}}{2\lambda}\right)\\
&+\frac{3(\alpha-1)}{N^{1+\tilde{\delta}}}\left(\mathbb{E}\left(\mathcal{H}((\tilde{Y}^{j,N}_0)_{j})\right)+Ne^{2\lambda t}+2N\sigma_N\frac{e^{2\lambda t-1}}{2\lambda}\right)\\
&+\frac{6(\alpha-1)}{\alpha(2-\alpha) }\frac{e^{2\lambda t-1}}{2\lambda}\left(\frac{1}{M^{\delta}}+\frac{1}{N^{\tilde{\delta}}}\right).
\end{align*}
We now choose the coefficients $\delta$ and $\tilde{\delta}$. Consider
\begin{align*}
\delta=\frac{2-\alpha}{\alpha}\frac{\ln N}{\ln M}\ \ \ \text{ and }\ \ \ \tilde{\delta}=\frac{2-\alpha}{\alpha}\frac{\ln M}{\ln N}.
\end{align*}
This way, we have both
\begin{align*}
\frac{M^{\delta(\alpha-1)}}{N^{2-\alpha}}=\frac{e^{\delta(\alpha-1)\ln M}}{N^{2-\alpha}}=\frac{e^{\frac{(2-\alpha)(\alpha-1)}{\alpha}\ln N}}{N^{2-\alpha}}=N^{-\frac{2-\alpha}{\alpha}}
&\text{ and, } 
\frac{N^{\tilde{\delta}(\alpha-1)}}{M^{2-\alpha}}=M^{-\frac{2-\alpha}{\alpha}},
\end{align*}
and
\begin{align*}
M^{-\delta}=M^{-\frac{2-\alpha}{\alpha}\frac{\ln N}{\ln M}}=e^{-\frac{2-\alpha}{\alpha}\ln N}=N^{-\frac{2-\alpha}{\alpha}}&\text{ and, likewise, }N^{-\tilde{\delta}}=M^{-\frac{2-\alpha}{\alpha}}.
\end{align*}
And thus
\begin{align*}
\mathbb{E}&\left(e^{2\lambda t}\mathcal{W}_2\left(\mu_t^N,\mu_t^M\right)^2\right)\\
\leq&\mathbb{E}\left(\mathcal{W}_2\left(\mu_0^N,\mu_0^M\right)^2\right)\\
&+3(\alpha-1)\left(\frac{1}{N^{\frac{2-\alpha}{\alpha}}}\mathbb{E}\left(\frac{1}{M}\mathcal{H}((\tilde{X}^{i,M}_0)_{i})\right)+\frac{1}{M^{\frac{2-\alpha}{\alpha}}}\mathbb{E}\left(\frac{1}{N}\mathcal{H}((\tilde{Y}^{j,N}_0)_{j})\right)\right)\\
&+\frac{e^{2\lambda t}-1}{2\lambda}\left(2(\sigma_M+\sigma_N)+\frac{2}{\alpha(2-\alpha)}\left(\frac{1}{N^{(2-\alpha)/\alpha}}+\frac{1}{M^{(2-\alpha)/\alpha}}\right)\right.\\
&\hspace{2cm}\left.+6(\alpha-1)\left(\frac{\sigma_N}{M^{(2-\alpha)/\alpha}}+\frac{\sigma_M}{N^{(2-\alpha)/\alpha}}\right)\right.\\
&\hspace{2cm}\left.+\frac{6(\alpha-1)}{\alpha(2-\alpha)}\left(\frac{1}{N^{(2-\alpha)/\alpha}}+\frac{1}{M^{(2-\alpha)/\alpha}}\right)\right)\\
&+3(\alpha-1)e^{2\lambda t}\left(\frac{1}{N^{(2-\alpha)/\alpha}}+\frac{1}{M^{(2-\alpha)/\alpha}}\right).
\end{align*}
This yields the result.
\end{proof}


\begin{lemma}
Consider Assumptions~\ref{Hyp_U_conv}~and~\ref{Hyp_V}, with $\alpha\in]1,2[$. Let $(\mu^N)_{N\in\mathbb{N}}$ be any sequence of independent empirical measures, such that $\mu_t^N$ is the empirical measure of the $N$ particle system at time $t$. There exist positive constants $C_1, C_2$ and $C_3$ such that for all $t\geq0$ and all $N,M\geq 1$
\begin{align}
\mathbb{E}&\left(\sup_{s\in[0,t]}\mathcal{W}_2\left(\mu_s^N,\mu_s^M\right)^2\right)
\\
\leq&e^{C_1 t}\left(\mathbb{E}\left(\mathcal{W}_2\left(\mu_0^N,\mu_0^M\right)^2\right)+3(\alpha-1)\left(\frac{\mathbb{E}^{\mu_0^M}\left(|X|^2\right)}{N^{(2-\alpha)/\alpha}}+\frac{\mathbb{E}^{\mu_0^N}\left(|X|^2\right)}{M^{(2-\alpha)/\alpha}}\right)\right.\nonumber\\
&\left.\hspace{2cm}+C_2(\sigma_N+\sigma_M)+C_3\left(\frac{1}{N^{(2-\alpha)/\alpha}}+\frac{1}{M^{(2-\alpha)/\alpha}}\right)\right)\label{non_unif_avec_sup_2}.
\end{align}
\end{lemma}


\begin{proof}
Consider a similar set up as the proof of Lemma~\ref{cauchy_alpha_1}, and define $(\tilde{X}^{i,M}_t)_i$, $(\tilde{Y}^{j,N}_t)_j$, $(X^{i}_t)_i$, $(Y^{j}_t)_j$ in the same manner. With similar calculations, Itô's formula yields
\begin{align*}
\mathcal{W}_2&\left(\mu_t^N,\mu_t^M\right)^2\\
\leq&\mathcal{W}_2\left(\mu_0^N,\mu_0^M\right)^2+\left(2(\sigma_M+\sigma_N)+\frac{2}{\alpha(2-\alpha)}\left(\frac{1}{N^{\frac{2-\alpha}{\alpha}}}+\frac{1}{M^{\frac{2-\alpha}{\alpha}}}\right)\right)t\\
&+\frac{6(\alpha-1)}{\alpha MN^{\frac{2-\alpha}{\alpha}}}\int_{0}^t\frac{1}{M^2}\sum_{i>j}\frac{i-j}{|\tilde{X}^{i,M}_s-\tilde{X}^{j,M}_s|^{\alpha}}ds\\
&+\frac{6(\alpha-1)}{\alpha NM^{\frac{2-\alpha}{\alpha}}}\int_{0}^t\frac{1}{N^2}\sum_{i>j}\frac{i-j}{|\tilde{Y}^{i,N}_s-\tilde{Y}^{j,N}_s|^{\alpha}}ds\\
&+\frac{2\sqrt{2\sigma_M}}{NM}\sum_i\int_0^t(X^i_s-Y^i_s)dB^i_s-\frac{2\sqrt{2\sigma_N}}{NM}\sum_i\int_0^t(X^i_s-Y^i_s)dB'^i_s.
\end{align*}
Similarly as Lemma~\ref{cauchy_alpha_1}, we use Burkholder-Davis-Gundy inequality to show that there exists a constant $C_{BDG}$ such that
\begin{align*}
\mathbb{E}&\left(\frac{2\sqrt{2\sigma_M}}{NM}\sup_{s\in[0,t]}\sum_i\int_0^s\left(X^i_u-Y^i_u\right)dB^{i}_u\right)\\
&\hspace{1cm}\leq2C_{BDG}\sigma_M+C_{BDG}\mathbb{E}\left(\int_0^t\mathcal{W}_2\left(\mu_s^N,\mu_s^M\right)^2ds\right),
\end{align*}
and
\begin{align*}
\mathbb{E}&\left(\frac{2\sqrt{2\sigma_N}}{NM}\sup_{s\in[0,t]}\sum_i-\int_0^s\left(X^i_u-Y^i_u\right)dB'^{i}_u\right)\\
&\hspace{1cm}\leq2C_{BDG}\sigma_N+C_{BDG}\mathbb{E}\left(\int_0^t\mathcal{W}_2\left(\mu_s^N,\mu_s^M\right)^2ds\right).
\end{align*}
We now use \eqref{borne_interaction_sup} to obtain
\begin{align*}
\mathbb{E}&\left(\sup_{s\in[0,t]}\mathcal{W}_2\left(\mu_s^N,\mu_s^M\right)^2\right)\\
\leq&\mathbb{E}\left(\mathcal{W}_2\left(\mu_0^N,\mu_0^M\right)^2\right)+\left(2(\sigma_M+\sigma_N)+\frac{2}{\alpha(2-\alpha)}\left(\frac{1}{N^{(2-\alpha)/\alpha}}+\frac{1}{M^{(2-\alpha)/\alpha}}\right)\right)t\\
&+\frac{6(\alpha-1)}{\alpha MN^{(2-\alpha)/\alpha}}\left(\frac{\alpha}{2}\left(\mathbb{E}\mathcal{H}(\textbf{X}_0)+M+\left(2M\sigma_M+M\right)t\right)+C(\alpha,M)t\right)\\
&+\frac{6(\alpha-1)}{\alpha NM^{(2-\alpha)/\alpha}}\left(\frac{\alpha}{2}\left(\mathbb{E}\mathcal{H}(\textbf{Y}_0)+N+\left(2N\sigma_N+N\right)t\right)+C(\alpha,N)t\right)\\
&+2C_{BDG}(\sigma_N+\sigma_M)+2C_{BDG}\mathbb{E}\left(\int_0^t\mathcal{W}_2\left(\mu_s^N,\mu_s^M\right)^2ds\right),
\end{align*}
and thus
\begin{align*}
\mathbb{E}&\left(\sup_{s\in[0,t]}\mathcal{W}_2\left(\mu_s^N,\mu_s^M\right)^2\right)\\
\leq&\mathbb{E}\left(\mathcal{W}_2\left(\mu_0^N,\mu_0^M\right)^2\right)+\frac{3(\alpha-1)}{ N^{(2-\alpha)/\alpha}}\mathbb{E}\left(\frac{1}{M}\mathcal{H}(\textbf{X}_0)\right)+\frac{3(\alpha-1)}{M^{(2-\alpha)/\alpha}}\mathbb{E}\left(\frac{1}{N}\mathcal{H}(\textbf{Y}_0)\right)\\
&+2C_{BDG}(\sigma_N+\sigma_M)+3(\alpha-1)\left(\frac{1}{N^{(2-\alpha)/\alpha}}+\frac{1}{M^{(2-\alpha)/\alpha}}\right)\\
&+\left(2(\sigma_M+\sigma_N)+6(\alpha-1)\left(\frac{\sigma_M}{N^{(2-\alpha)/\alpha}}+\frac{\sigma_N}{M^{(2-\alpha)/\alpha}}\right)\right.\\
&\hspace{0.5cm}\left.+\frac{6\alpha-4}{\alpha(2-\alpha)}\left(\frac{1}{N^{(2-\alpha)/\alpha}}+\frac{1}{M^{(2-\alpha)/\alpha}}\right)\right.\\
&\hspace{0.5cm}\left.+3\lambda(\alpha-1)\left(\frac{1}{N^{(2-\alpha)/\alpha}}+\frac{1}{M^{(2-\alpha)/\alpha}}\right)\right)t\\
&+2C_{BDG}\mathbb{E}\left(\int_0^t\sup_{u\in[0,s]}\mathcal{W}_2\left(\mu_u^N,\mu_u^M\right)^2ds\right).
\end{align*}
Denote
\begin{align*}
C_{\text{prop}}(N,M,\alpha):=&\left(\frac{6\alpha-4}{\alpha(2-\alpha)}+3\lambda(\alpha-1)\right)\left(\frac{1}{N^{(2-\alpha)/\alpha}}+\frac{1}{M^{(2-\alpha)/\alpha}}\right)\\
&+6(\alpha-1)\left(\frac{\sigma_M}{N^{(2-\alpha)/\alpha}}+\frac{\sigma_N}{M^{(2-\alpha)/\alpha}}\right)+2(\sigma_M+\sigma_N),\\
D_{\text{prop}}(N,M,\alpha):=&\mathbb{E}\left(\mathcal{W}_2\left(\mu_0^N,\mu_0^M\right)^2\right)+3(\alpha-1)\left(\frac{\mathbb{E}\left(\frac{1}{M}\mathcal{H}(\textbf{X}_0)\right)}{N^{(2-\alpha)/\alpha}}+\frac{\mathbb{E}\left(\frac{1}{N}\mathcal{H}(\textbf{Y}_0)\right)}{M^{(2-\alpha)/\alpha}}\right)\\
&+2C_{BDG}(\sigma_N+\sigma_M)+3(\alpha-1)\left(\frac{1}{N^{(2-\alpha)/\alpha}}+\frac{1}{M^{(2-\alpha)/\alpha}}\right),
\end{align*}
such that, for the sake of conciseness, we have
\begin{align*}
\mathbb{E}&\left(\sup_{s\in[0,t]}\mathcal{W}_2\left(\mu_s^N,\mu_s^M\right)^2\right)\\
&\leq D_{\text{prop}}(N,M,\alpha)+2C_{BDG}\mathbb{E}\left(\int_0^t\left(\sup_{u\in[0,s]}\mathcal{W}_2\left(\mu_u^N,\mu_u^M\right)^2+\frac{C_{\text{prop}}(N,M,\alpha)}{2C_{BDG}}\right)ds\right).
\end{align*}
Using Gronwall lemma on $t\mapsto\mathbb{E}\left(\sup_{s\in[0,t]}\mathcal{W}_2\left(\mu_s^N,\mu_s^M\right)^2\right)+\frac{C_{\text{prop}}(N,M,\alpha)}{4C_{BDG}}$, we get for all $t\geq0$
\begin{align*}
\mathbb{E}\left(\sup_{s\in[0,t]}\mathcal{W}_2\left(\mu_s^N,\mu_s^M\right)^2\right)&+\frac{C_{\text{prop}}(N,M,\alpha)}{2C_{BDG}}\\
&\leq e^{2C_{BDG}t}\left(D_{\text{prop}}(N,M,\alpha)+\frac{C_{\text{prop}}(N,M,\alpha)}{2C_{BDG}}\right).
\end{align*}

\end{proof}

%
%

\subsection{Conclusion}

We now wish to prove that the Cauchy-like estimates \eqref{unif_sans_sup} and \eqref{unif_sans_sup_2} are sufficient to conclude on the convergence, at any given $t>0$, of the empirical measures.


\begin{lemma}\label{conv_a_t_fixe}
For any sequence $(\mu^n)_{n\in\mathbb{N}}$ of independent random measures in $\mathcal{P}_2\left(\mathbb{R}\right)$, if 
\begin{equation}\label{Cauchy_seq_point}
\forall \epsilon>0,\ \exists N\geq0,\ \forall n,m\geq N,\  \mathbb{E}\mathcal{W}_2\left(\mu^n,\mu^m\right)\leq\epsilon,
\end{equation}
then there exists a deterministic measure $\rho\in\mathcal{P}_2\left(\mathbb{R}\right)$ such that 
$$\mathbb{E}\mathcal{W}_2\left(\mu^n,\rho\right)\rightarrow 0\ \ \ \text{ as }\ \ \ n\rightarrow\infty.$$
\end{lemma}


\begin{proof}
Let us start by mentioning the result of \cite{bolley_separability}, which states that if $(X,d)$ is a complete metric space, then so is $(\mathcal{P}(X),\mathcal{W}_d)$, where $\mathcal{W}_d$ is the Wasserstein distance associated to $d$. Denote, for $\xi$ and $\zeta$ two probability measures on the space $\mathcal{P}_2(\mathbb{R})$ and $\Gamma$ the set of couplings of $\xi$ and $\zeta$, the Wasserstein distance
\begin{equation}
\mathbb{W}(\xi,\zeta)=\inf_{(\mu,\nu)\sim\Gamma}\mathbb{E}\mathcal{W}_2(\mu,\nu).
\end{equation}
The metric space $\left(\mathcal{P}_1(\mathcal{P}_2(\mathbb{R})),\mathbb{W}\right)$ is complete.

Let $\xi^n$ be the law of $\mu^n$. The assumption \eqref{Cauchy_seq_point} implies, since $\mathbb{W}(\xi^n,\xi^m)\leq\mathbb{E}\mathcal{W}_2\left(\mu^n,\mu^m\right)$, that there exists a measure $\zeta\in\mathcal{P}(\mathcal{P}_2(\mathbb{R}))$ such that
$$\mathbb{W}(\xi^n,\zeta)\rightarrow 0\ \ \ \text{ as }\ \ \ n\rightarrow\infty.$$
Denote by $\pi^n$ the optimal coupling between $\xi^n$ and $\zeta$ for the Wasserstein distance above. Considering $\pi^1\otimes\pi^2\otimes...$, there exists a sequence $(\rho^n)_n$, of independent measures identically distributed according to $\zeta$, such that $$\mathbb{E}\mathcal{W}_2\left(\mu^n,\rho^n\right)\rightarrow 0\ \ \ \text{ as }\ \ \ n\rightarrow\infty.$$ 

We now wish to prove that all $\rho^n$ are almost surely equal. To do so, we will make use of the assumption of independence of the sequence $\mu^n$. We have 
\begin{align*}
&\forall \epsilon>0,\ \exists N\geq0,\ \forall n\geq N,\ \forall p>0,\   \mathbb{E}\mathcal{W}_2\left(\mu^n,\mu^{n+p}\right)\leq\epsilon,\\
&\forall \epsilon>0,\ \exists N\geq0,\ \forall n\geq N,\  \mathbb{E}\mathcal{W}_2\left(\mu^n,\rho^n\right)\leq\epsilon.
\end{align*}
Direct triangle inequalities using the two assertions above yield
\begin{align*}
\forall \epsilon>0,\ \exists N\geq0,\ \forall n\geq N,\ \forall p>0,\  \mathbb{E}\mathcal{W}_2\left(\mu^n,\rho^{n+p}\right)\leq\epsilon,\\
\forall \epsilon>0,\ \exists N\geq0,\ \forall n\geq N,\ \forall p>0,\  \mathbb{E}\mathcal{W}_2\left(\rho^n,\rho^{n+p}\right)\leq\epsilon.
\end{align*}
The fact that $\mathbb{E}\mathcal{W}_2\left(\mu^n,\rho^n\right)\rightarrow 0$ implies $\mathbb{E}\mathcal{W}_1\left(\mu^n,\rho^n\right)\rightarrow 0$. The dual formulation of the $L^1$ Wasserstein distance yields
\begin{align*}
\mathbb{E}\sup_{||\psi||_{Lip}\leq1}\left|\mu^n(\psi)-\rho^n(\psi)\right|\rightarrow 0.
\end{align*}
Let $f$ be a bounded Lipschitz continuous function. We have
\begin{align}
\mathbb{E}\left|\mu^n(f)-\rho^n(f)\right|\rightarrow0\label{Convergence_lip}.
\end{align}
In particular we get $\mathbb{E}\mu^n(f)\rightarrow\mathbb{E}\rho(f)$, with $\rho\sim\zeta$. Likewise
\begin{align}
\mathbb{E}\left|\rho^n(f)-\rho^{n+1}(f)\right|\rightarrow0\label{Convergence_lip_2}.
\end{align}
On the one hand, using the independence of the sequence,
\begin{align*}
\mathbb{E}\left(\mu^{n}(f)\mu^{n+1}(f)\right)=\mathbb{E}\left(\mu^{n}(f)\right)\mathbb{E}\left(\mu^{n+1}(f)\right)\rightarrow \mathbb{E}\left(\rho(f)\right)^2.
\end{align*}
On the other hand
\begin{align*}
\mathbb{E}&\left(\mu^{n}(f)\mu^{n+1}(f)\right)\\
&=\mathbb{E}\left(\left(\rho^n(f)\right)^2+\rho^n(f)\left(\mu^n(f)-\rho^n(f)\right)+\mu^n(f)\left(\mu^{n+1}(f)-\rho^{n+1}(f)\right)\right.\\
&\hspace{1cm}\left.+\mu^n(f)\left(\rho^{n+1}(f)-\rho^{n}(f)\right)\right).
\end{align*}
Let us consider each term individually
\begin{align*}
\mathbb{E}\left(\left(\rho^n(f)\right)^2\right)=&\mathbb{E}\left(\left(\rho(f)\right)^2\right),\\
\mathbb{E}\left(\rho^n(f)\left(\mu^n(f)-\rho^n(f)\right)\right)\leq&||f||_{\infty}\mathbb{E}\left|\mu^n(f)-\rho^n(f)\right|\rightarrow0\text{ using \eqref{Convergence_lip}},\\
\mathbb{E}\left(\mu^n(f)\left(\mu^{n+1}(f)-\rho^{n+1}(f)\right)\right)\leq&||f||_{\infty}\mathbb{E}\left|\mu^{n+1}(f)-\rho^{n+1}(f)\right|\rightarrow0\text{ using \eqref{Convergence_lip}},\\
\mathbb{E}\left(\mu^n(f)\left(\rho^{n+1}(f)-\rho^{n}(f)\right)\right)\leq&||f||_{\infty}\mathbb{E}\left|\rho^{n+1}(f)-\rho^{n}(f)\right|\rightarrow0\text{ using \eqref{Convergence_lip_2}}.
\end{align*}
Thus 
\begin{align*}
\mathbb{E}&\left(\mu^{n}(f)\mu^{n+1}(f)\right)\rightarrow\mathbb{E}\left(\left(\rho(f)\right)^2\right).
\end{align*}
We have obtained
\begin{align*}
\mathbb{E}\left(\rho(f)\right)^2=\mathbb{E}\left(\left(\rho(f)\right)^2\right),
\end{align*}
which implies that for any bounded and Lipschitz continuous function $f$, $\rho(f)$ is almost surely constant. Let $\rho_1$ and $\rho_2$ be two random variables with law $\zeta$, considering two random variables $X\sim\rho_1$ and $Y\sim\rho_2$, we get for all Lipschitz continuous bounded function $h$ that $\mathbb{E}h(X)=\mathbb{E}h(Y)$. Let $a<b$ be two real numbers. Consider 
\begin{align*}
g_m(x)=&\left\{
\begin{array}{ll}
1 &\text{ if }x\in[a+\frac{1}{m},b-\frac{1}{m}]\\
0 &\text{ if }x\leq a\text{ or }x\geq b\\
m(x-a)&\text{ if }a<x<a+\frac{1}{m}\\
m(b-x)&\text{ if }b-\frac{1}{m}<x<b
\end{array}
\right.
\end{align*}
By construction, $(g_m)_{m\in\mathbb{N}}$ is an increasing sequence of bounded Lipschitz continuous functions such that for all $m$,  $g_m\leq\mathds{1}_{]a,b[}$ and for all $x\in\mathbb{R}$, $g_m(x)\rightarrow\mathds{1}_{]a,b[}(x)$ as $m\rightarrow\infty$. We thus have for all $m\in\mathbb{N}$ the equality $\mathbb{E}g_m(X)=\mathbb{E}g_m(Y)$ and,  by the monotone convergence theorem, $\mathbb{E}\mathds{1}_{]a,b[}(X)=\mathbb{E}\mathds{1}_{]a,b[}(Y)$. Then, again by the monotone convergence theorem and by considering an increasing sequence of simple functions, the equality $\mathbb{E}h(X)=\mathbb{E}h(Y)$ holds true for all bounded measurable function $h$. The variables $X$ and $Y$ are thus equal in law, and $\rho$ is therefore a deterministic probability measure.

\end{proof}

Lemma~\ref{conv_a_t_fixe} allows us to conclude on the convergence, at any given $t\geq0$, of the sequence of empirical measures $(\mu^N_t)_N$ towards $\bar{\rho}_t$ where, at least formally, $\bar{\rho}$ is a solution of the non linear limit equation (we refer to the next Section~\ref{appendice_id_limit} for a more rigorous identification of the equation satisfied by the limit $\bar{\rho}$). However, \textit{a priori}, if the limit equation admits several solutions, nothing guarantees that the sequence converges towards the \textit{same} solution at two different times $t_1$ and $t_2$. To show this, we now use the estimates \eqref{non_unif_avec_sup} and \eqref{non_unif_avec_sup_2} which, even though on their own do not ensure uniform in time convergence because of the exponential term, show that on any time interval $[0,T]$ there is uniform convergence towards a unique solution of the limit equation. This result, combined with the uniform in time pointwise convergence given by Lemma~\ref{conv_a_t_fixe}, will yield the desired result.
Denote $\mathcal{C}([0,T],\mathcal{P}_2(\mathbb{R}))$ the space of continuous functions taking values in the space of probability measures $\mathcal{P}_2(\mathbb{R})$ endowed with the $L^2$ Wasserstein distance.


\begin{lemma}
Let $T\geq0$. For any sequence $(\mu^n)_{n\in\mathbb{N}}$ of independent random variables in $\mathcal{C}([0,T],\mathcal{P}_2(\mathbb{R}))$ , if 
\begin{equation}\label{Cauchy_seq}
\forall \epsilon>0,\ \exists N\geq0,\ \forall n,m\geq N,\  \mathbb{E}\sup_{t\in[0,T]}\mathcal{W}_2\left(\mu_t^n,\mu_t^m\right)\leq\epsilon,
\end{equation}
then there exists a deterministic measure $(\rho_t)_{t\in[0,T]}\in\mathcal{C}([0,T],\mathcal{P}_2(\mathbb{R}))$ such that
$$\mathbb{E}\sup_{t\in[0,T]}\mathcal{W}_2\left(\mu_t^n,\rho_t\right)\rightarrow 0\ \ \ \text{ as }\ \ \ n\rightarrow\infty.$$
\end{lemma}


\begin{proof}
Let, for $\xi$ and $\zeta$ two probability measures on the space $\mathcal{C}([0,T],\mathcal{P}_2(\mathbb{R}))$ and $\Gamma$ the set of couplings of $\xi$ and $\zeta$, the Wasserstein distance be defined by
\begin{equation}
\mathbb{W}_s(\xi,\zeta)=\inf_{(\mu,\nu)\sim\Gamma}\mathbb{E}\left(\sup_{t\in[0,T]}\mathcal{W}_2(\mu_t,\nu_t)\right).
\end{equation}
Let $\xi^n$ be the law of $\mu^n$. By completeness,   Assumption~\eqref{Cauchy_seq} implies that there exists a probability measure $\zeta$ on $\mathcal{C}([0,T],\mathcal{P}_2(\mathbb{R}))$ such that
$$\mathbb{W}_s(\xi^n,\zeta)\rightarrow 0\ \ \ \text{ as }\ \ \ n\rightarrow\infty.$$
Thus, there exists a sequence $(\rho^n)_n$, identically distributed according to $\zeta$, such that $$\mathbb{E}\sup_{t\in[0,T]}\mathcal{W}_2\left(\mu^n_t,\rho^n_t\right)\rightarrow 0\ \ \ \text{ as }\ \ \ n\rightarrow\infty.$$ 
By the same proof as Lemma~\ref{conv_a_t_fixe}, we get that for all $t\geq0$, all $\rho^n_t$ are almost surely equal, hence the result.
\end{proof}

%
%

\section{Identification of the limit}\label{appendice_id_limit}

The goal of this subsection is to identify, in a more rigorous way than the formal calculations of the introduction, the limit $\bar{\rho}_t$, and more precisely the PDE it satisfies. The goal is to prove the following theorem


\begin{theorem}\label{id_limit}
For $\alpha\in]1,2[$ under Assumptions~\ref{Hyp_U_conv}~and~\ref{Hyp_V}, or for $\alpha=1$ under Assumptions~\ref{Hyp_U_lip}~and~\ref{Hyp_V}, both with $\sigma_N\rightarrow0$, the limit $(\bar{\rho}_t)_{t\geq0}$ of the sequence of empirical measures $((\mu_t^N)_{t\geq0})_{N\geq2}$ satisfies, for all functions $f\in\mathcal{C}^2(\mathbb{R})$ with bounded derivatives such that $f$, $f'$, $f'U'$, and $f''$ are Lipschitz continuous and that $f'U'$ is bounded, the following equation: for all $t\geqslant 0$,
\begin{align*}
\int_{\mathbb{R}}f(x)\bar{\rho}_t(dx)=&\int_{\mathbb{R}}f(x)\bar{\rho}_0(dx)-\int_{0}^t\int_{\mathbb{R}}f'(x)U'(x)\bar{\rho}_s(dx)ds\\
&+\frac{1}{2}\int_{0}^t\int\int_{\{x\neq y\}}\frac{(f'(x)-f'(y))(x-y)}{|x-y|^{\alpha+1}}\bar{\rho}_s(dx)\bar{\rho}_s(dy)ds.
\end{align*}
\end{theorem}

To do so, let us first mention the following lemma, which is a consequence of previous calculations.


\begin{lemma}
For $\alpha\in]1,2[$ and under Assumptions~\ref{Hyp_U_conv}~and~\ref{Hyp_V}, for all $t\geqslant 0$, there exists a constant $C_{int}$ such that for all $N\geq2$, we have the following estimates
\begin{equation}\label{estime_int_id_lim}
\mathbb{E}\left(\int_{0}^t\int\int_{\{x\neq y\}}\frac{1}{|x-y|^{\frac{(\alpha-1)(\alpha+2)}{2\alpha}}}\mu^N_s(dx)\mu^N_s(dy)ds\right)\leq C_{int}.
\end{equation}
\end{lemma}


\begin{proof}
Let $(X^1_t,...,X^N_t)_t$ be the unique strong solution of \eqref{particle_system}, and $\mu^N_t$ the associated empirical measure. By definition
\begin{align*}
\mathbb{E}&\left(\int_{0}^t\int\int_{\{x\neq y\}}\frac{1}{|x-y|^{\frac{(\alpha-1)(\alpha+2)}{2\alpha}}}\mu^N_s(dx)\mu^N_s(dy)ds\right)\\
&=2\mathbb{E}\left(\int_{0}^t\frac{1}{N^2}\sum_{i> j}\frac{1}{|X^i_t-X^j_t|^{\frac{(\alpha-1)(\alpha+2)}{2\alpha}}}ds\right).
\end{align*}
Young's inequality yields, for $i>j$, for $\beta>0$, $\gamma>0$, and $p>1$ and $q>1$ such that $\frac{1}{p}+\frac{1}{q}=1$
\begin{align*}
\frac{1}{|X^i_t-X^j_t|^{\frac{(\alpha-1)(\alpha+2)}{2\alpha}}}\leq \frac{\gamma^p}{p}\left(\frac{(i-j)^\beta}{|X^i_t-X^j_t|^{\frac{(\alpha-1)(\alpha+2)}{2\alpha}}}\right)^p+\frac{1}{q\gamma^q}\frac{1}{(i-j)^{\beta q}}.
\end{align*}
We choose
\begin{align*}
\beta=\frac{\alpha-1}{\alpha},\qquad p=\frac{\alpha}{\alpha-1},\qquad  q=\alpha,\qquad  \gamma=N^{-\frac{\alpha-1}{\alpha}},
\end{align*}
which yields
\begin{align*}
\frac{1}{|X^i_t-X^j_t|^{\frac{(\alpha-1)(\alpha+2)}{2\alpha}}}\leq\frac{\alpha-1}{\alpha}\frac{1}{N}\frac{i-j}{|X^i_t-X^j_t|^{\frac{\alpha+2}{2}}}+\frac{N^{\alpha-1}}{\alpha}\frac{1}{(i-j)^{\alpha-1}}.
\end{align*}
We have
\begin{align*}
\sum_{i>j}\frac{1}{(i-j)^{\alpha-1}}=\sum_{i=1}^N\sum_{j=1}^{i-1}\frac{1}{(i-j)^{\alpha-1}}=\sum_{i=1}^N\sum_{j=1}^{i-1}\frac{1}{j^{\alpha-1}}\leq N\sum_{j=1}^{N}\frac{1}{j^{\alpha-1}}\leq\frac{N}{2-\alpha}N^{2-\alpha},
\end{align*}
where this last inequality comes from Lemma~\ref{comp_serie_int}. Thus
\begin{align*}
\int_0^t\frac{1}{N^2}\sum_{i>j}\frac{1}{|X^i_s-X^j_s|^{\frac{(\alpha-1)(\alpha+2)}{2\alpha}}}ds\leq&\frac{\alpha-1}{\alpha}\frac{1}{N}\left(\int_0^t\frac{1}{N^2}\sum_{i>j}\frac{i-j}{|X^i_s-X^j_s|^{\frac{\alpha+2}{2}}}ds\right)\\
&+\frac{1}{N^2}\frac{N^{\alpha-1}}{\alpha}\frac{N}{2-\alpha}N^{2-\alpha}t.
\end{align*}
This yields the result using \eqref{borne_interaction_sup}, as $\frac{\alpha+2}{2}\in]1,2[$, and noticing that $\frac{1}{N}\mathbb{E}\mathcal{H}(X_0)$ is bounded from above by the initial second moment, and is thus bounded uniformly in $N$.
\end{proof}


\noindent\textit{Proof of Theorem~\ref{id_limit}.} 
As $\mathbb{E}\mathcal{W}_1(\mu^N_t,\bar{\rho}_t)\rightarrow0$, we get by the dual formulation of the Wasserstein distance that for all function $g$ Lipshitz continuous
\begin{align*}
\mathbb{E}\int_{\mathbb{R}}g(x)\mu^N_t(dx)\rightarrow\int_{\mathbb{R}}g(x)\bar{\rho}_t(dx).
\end{align*}
Likewise, since
\begin{align*}
\mathcal{W}_1(\mu^N_t\otimes\mu^N_t,\bar{\rho}_t\otimes\bar{\rho}_t)\leq&\mathcal{W}_1(\mu^N_t\otimes\mu^N_t,\bar{\rho}_t\otimes\mu^N_t)+\mathcal{W}_1(\bar{\rho}_t\otimes\mu^N_t,\bar{\rho}_t\otimes\bar{\rho}_t)\\
=&\inf_{\begin{array}{ll}(X^1,X^2)\sim\mu^N_t\otimes\mu^N_t\\(Y^1,Y^2)\sim\bar{\rho}_t\otimes\mu^N_t\end{array}}\mathbb{E}\left(|X^1-Y^1|+|X^2-Y^2|\right)\\
&+\inf_{\begin{array}{ll}(X^1,X^2)\sim\bar{\rho}_t\otimes\mu^N_t\\(Y^1,Y^2)\sim\bar{\rho}_t\otimes\bar{\rho}_t\end{array}}\mathbb{E}\left(|X^1-Y^1|+|X^2-Y^2|\right)\\
\leq&\inf_{X^1\sim\mu^N_t, Y^1\sim\bar{\rho}_t}\mathbb{E}\left(|X^1-Y^1|\right)+\inf_{X^2\sim\mu^N_t, Y^2\sim\bar{\rho}_t}\mathbb{E}\left(|X^2-Y^2|\right)\\
=&2\mathcal{W}_1(\mu^N_t,\bar{\rho}_t),
\end{align*}
we get that for all function $g$ Lipshitz continuous
\begin{align*}
\mathbb{E}\int_{\mathbb{R}}\int_{\mathbb{R}}g(x,y)\mu^N_t(dx)\mu^N_t(dy)\rightarrow\int_{\mathbb{R}}\int_{\mathbb{R}}g(x,y)\bar{\rho}_t(dx)\bar{\rho}_t(dy).
\end{align*}
Let us now consider a function $f\in\mathcal{C}^2(\mathbb{R})$ with bounded derivatives such that $f$, $f'$,  $f'U'$ and $f''$ are Lipschitz continuous and $f'U'$ is bounded. By Itô's formula, we have
\begin{align*}
\int_{\mathbb{R}}f(x)\mu^N_t(dx)=&\int_{\mathbb{R}}f(x)\mu^N_0(dx)-\int_{0}^t\int_{\mathbb{R}}f'(x)U'(x)\mu^N_s(dx)ds\\
&+\int_0^t\int_{\mathbb{R}}\sigma_Nf''(x)\mu^N_s(dx)ds+\int_0^t\frac{\sqrt{2\sigma_N}}{N}\sum_{i=1}^Nf'(X^i_s)dB^i_s\\
&+\frac{1}{2}\int_{0}^t\int\int_{\{x\neq y\}}\frac{(f'(x)-f'(y))(x-y)}{|x-y|^{\alpha+1}}\mu^N_s(dx)\mu^N_s(dy)ds\\
:=&I_0(N)-I_1(N)+I_2(N)+I_3(N)+I_4(N).
\end{align*}
Let us deal with each terms. 
\begin{description}
\item[$\bullet\  I_0(N)$ :]Since we assume $f$ to be Lipschitz continuous
\begin{align*}
\mathbb{E}I_0(N)=\mathbb{E}\int_{\mathbb{R}}f(x)\mu^N_0(dx)\longrightarrow\int_{\mathbb{R}}f(x)\bar{\rho}_0(dx).
\end{align*}

\item[$\bullet\  I_1(N)$ :]$f'U'$ being Lipschitz continuous, we have
\begin{align*}
\mathbb{E}\int_{\mathbb{R}}f'(x)U'(x)\mu^N_s(dx)\rightarrow\int_{\mathbb{R}}f'(x)U'(x)\bar{\rho}_s(dx).
\end{align*}
Furthermore, since $f'U'$ is bounded,  $\left|\int_{\mathbb{R}}f'(x)U'(x)\mu^N_s(dx)\right|\leq||f'U'||_{\infty}$ and we have by dominated convergence
\begin{align*}
\mathbb{E}I_1(N)=\mathbb{E}\int_{0}^t\int_{\mathbb{R}}f'(x)U'(x)\mu^N_s(dx)ds&=\int_{0}^t\mathbb{E}\int_{\mathbb{R}}f'(x)U'(x)\mu^N_s(dx)ds\\
&\longrightarrow\int_{0}^t\int_{\mathbb{R}}f'(x)U'(x)\bar{\rho}_s(dx)ds
\end{align*}

\item[$\bullet\  I_2(N)$ :]Since we assume $f''$ to be Lipschitz continous
\begin{align*}
\mathbb{E}I_2(N)=\mathbb{E}\int_0^t\int_{\mathbb{R}}\sigma_Nf''(x)\mu^N_s(dx)ds&=\int_0^t\sigma_N\mathbb{E}\left(\int_{\mathbb{R}}f''(x)\mu^N_s(dx)\right)ds\\
&\longrightarrow0\ \ \ \text{ (by dominated convergence).}
\end{align*}

\item[$\bullet\  I_3(N)$ :]As $f'$ is bounded, $I_3(N)$ is a true martingale, and thus $\mathbb{E}I_3(N)=0$.

\item[$\bullet\  I_4(N)$ :] Let, for $R>0$, $\phi_R$ be a Lipshitz continous function such that
\begin{align*}
\phi_R(x)=\left\{\begin{array}{ll}1\ \ \ &\text{ if }x\leq R\\\frac{2R-x}{R}\ \ \ &\text{ if }R\leq x\leq 2R\\0\ \ \ &\text{ if }x\geq 2R. \end{array}\right.
\end{align*}
We have
\begin{align}
&\int_{0}^t\int\int_{\{x\neq y\}}\frac{(f'(x)-f'(y))(x-y)}{|x-y|^{\alpha+1}}\mu^N_s(dx)\mu^N_s(dx)(dy)ds\nonumber\\
&=\int_{0}^t\int\int_{\{x\neq y\}}\frac{(f'(x)-f'(y))(x-y)}{|x-y|^{\alpha+1}}\phi_R(|x-y|)\mu^N_s(dx)\mu^N_s(dx)(dy)ds\nonumber\\
&\ +\int_{0}^t\int\int\frac{(f'(x)-f'(y))(x-y)}{|x-y|^{\alpha+1}}(1-\phi_R(|x-y|))\mu^N_s(dx)\mu^N_s(dx)(dy)ds\label{inter_a_controler_pour_lim}.
\end{align}
Let us now find the limit as $R$ goes to 0 of the limit as $N$ goes to infinity of the expectation of the  first term of \eqref{inter_a_controler_pour_lim}. By Hölder's inequality
\begin{align*}
\mathbb{E}&\int_{0}^t\int\int_{\{x\neq y\}}\left|\frac{(f'(x)-f'(y))(x-y)}{|x-y|^{\alpha+1}}\phi_R(|x-y|)\right|\mu^N_s(dx)\mu^N_s(dx)(dy)ds\\
\leq&||f''||_\infty\mathbb{E}\left(\int_{0}^t\int\int_{\{x\neq y\}}\frac{1}{|x-y|^{\frac{(\alpha-1)(\alpha+2)}{2\alpha}}}\mu^N_s(dx)\mu^N_s(dx)(dy)ds\right)^{\frac{2\alpha}{\alpha+2}}\\
&\hspace{0.3cm}\times\mathbb{E}\left(\int_{0}^t\int\int\phi_R(|x-y|)^{\frac{\alpha+2}{2-\alpha}}\mu^N_s(dx)\mu^N_s(dx)(dy)ds\right)^{\frac{2-\alpha}{\alpha+2}},
\end{align*}
and since $0\leq\phi_R\leq1$,
\begin{align*}
&\mathbb{E}\left(\int_{0}^t\int\int\phi_R(|x-y|)^{\frac{\alpha+2}{2-\alpha}}\mu^N_s(dx)\mu^N_s(dx)(dy)ds\right)^{\frac{2-\alpha}{\alpha+2}}
\\
&\ \ \ \leq\mathbb{E}\left(\int_{0}^t\int\int\phi_R(|x-y|)\mu^N_s(dx)\mu^N_s(dx)(dy)ds\right)^{\frac{2-\alpha}{\alpha+2}}.
\end{align*}
We now use \eqref{estime_int_id_lim} to get
\begin{align}
\mathbb{E}&\int_{0}^t\int\int_{\{x\neq y\}}\frac{(f'(x)-f'(y))(x-y)}{|x-y|^{\alpha+1}}\phi_R(|x-y|)\mu^N_s(dx)\mu^N_s(dx)(dy)ds\nonumber\\
\leq&||f''||_\infty C_{int}^{\frac{2\alpha}{\alpha+2}}\mathbb{E}\left(\int_{0}^t\int\int\phi_R(|x-y|)\mu^N_s(dx)\mu^N_s(dx)(dy)ds\right)^{\frac{2-\alpha}{\alpha+2}}\label{inter_a_controler_pour_lim_1}.
\end{align}
We then use
\begin{align*}
\int_{0}^t\int\int&\phi_R(|x-y|)\mu^N_s(dx)\mu^N_s(dx)(dy)ds\\
\leq&\int_{0}^t\int\int_{\{x=y\}}\mu^N_s(dx)\mu^N_s(dx)(dy)ds\\
&+\int_{0}^t\int\int_{\{x\neq y\}}\phi_R(|x-y|)\mu^N_s(dx)\mu^N_s(dx)(dy)ds.
\end{align*}
First
\begin{align*}
\int_{0}^t\int\int_{\{x=y\}}\mu^N_s(dx)\mu^N_s(dx)(dy)ds=\frac{t}{N}.
\end{align*}
Then $\phi_R(|x|)\leq \mathds{1}_{|x|\leq2R}\leq\left(\frac{2R}{|x|}\right)^{\frac{(\alpha-1)(\alpha+2)}{2\alpha}}$,
which implies
\begin{align*}
\mathbb{E}&\left(\int_{0}^t\int\int_{\{x\neq y\}}\phi_R(|x-y|)\mu^N_s(dx)\mu^N_s(dx)(dy)ds\right)\\
&\leq(2R)^{\frac{(\alpha-1)(\alpha+2)}{2\alpha}}\mathbb{E}\left(\int_{0}^t\int\int_{\{x\neq y\}}\frac{1}{|x-y|^{\frac{(\alpha-1)(\alpha+2)}{2\alpha}}}\mu^N_s(dx)\mu^N_s(dx)(dy)ds\right)\\
&\leq (2R)^{\frac{(\alpha-1)(\alpha+2)}{2\alpha}}C_{int}.
\end{align*}
Thus, 
\begin{align}
\mathbb{E}\int_{0}^t\int\int\phi_R(|x-y|)\mu^N_s(dx)\mu^N_s(dx)(dy)ds\leq&\frac{t}{N}+(2R)^{\frac{(\alpha-1)(\alpha+2)}{2\alpha}}C_{int}\label{inter_a_controler_pour_lim_2}.
\end{align}
Thus, for the first term of \eqref{inter_a_controler_pour_lim}, using \eqref{inter_a_controler_pour_lim_1} and \eqref{inter_a_controler_pour_lim_2}, taking the limit as $N\rightarrow\infty$  and then as $R\rightarrow0$ yields
\begin{align}
\lim_{R\rightarrow0}\lim_{N\rightarrow\infty}&\mathbb{E}\int_{0}^t\int\int_{\{x\neq y\}}\frac{(f'(x)-f'(y))(x-y)}{|x-y|^{\alpha+1}}\phi_R(|x-y|)\mu^N_s(dx)\mu^N_s(dx)(dy)ds=0\label{I_4_terme1}.
\end{align}
\\

Let us find the limit as $R$ goes to 0 of the limit as $N$ goes to infinity of the expectation of the second term of \eqref{inter_a_controler_pour_lim}. Since $\frac{(f'(x)-f'(y))(x-y)}{|x-y|^{\alpha+1}}(1-\phi_R(|x-y|))$ is bounded and Lipschitz continous, we have
\begin{align*}
\mathbb{E}&\int_{0}^t\int\int\frac{(f'(x)-f'(y))(x-y)}{|x-y|^{\alpha+1}}(1-\phi_R(|x-y|))\mu^N_s(dx)\mu^N_s(dy)ds\\
&\longrightarrow\int_{0}^t\int\int\frac{(f'(x)-f'(y))(x-y)}{|x-y|^{\alpha+1}}(1-\phi_R(|x-y|))\bar{\rho}_s(dx)\bar{\rho}_s(dy)ds.
\end{align*}
We now want to use dominated convergence to consider the limit as $R$ goes to 0. We have
\begin{align*}
\left|\frac{(f'(x)-f'(y))(x-y)}{|x-y|^{\alpha+1}}(1-\phi_R(|x-y|))\right|\leq ||f''||_{\infty} \frac{\mathds{1}_{x\neq y}}{|x-y|^{\alpha-1}}.
\end{align*}
Let us show that $\int_{0}^t\int\int\frac{\mathds{1}_{x\neq y}}{|x-y|^{\alpha-1}}\bar{\rho}_s(dx)\bar{\rho}_s(dy)ds<\infty$. Using \eqref{estime_int_id_lim}, and Young's inequality as ${\alpha-1\leq\frac{(\alpha-1)(\alpha+2)}{2\alpha}}$, we get
\begin{align*}
\mathbb{E}&\int_{0}^t\int\int\frac{1-\phi_R(|x-y|)}{|x-y|^{\alpha-1}}\mu^N_s(dx)\mu^N_s(dy)ds\\
&\leq \mathbb{E}\int_{0}^t\int\int\frac{\mathds{1}_{x\neq y}}{|x-y|^{\alpha-1}}\mu^N_s(dx)\mu^N_s(dy)ds\leq \tilde{C}_{int},
\end{align*}
where $\tilde{C}_{int}$ is a constant independent of $N$ (depending on $C_{int}$). The righthand side being independent of $N$ and $R$, and since $\frac{1-\phi_R(|x-y|)}{|x-y|^{\alpha-1}}$ is bounded and Lipschtiz continous, we have taking the limit as $N\rightarrow\infty$
\begin{align*}
\mathbb{E}\int_{0}^t\int\int\frac{1-\phi_R(|x-y|)}{|x-y|^{\alpha-1}}\bar{\rho}_s(dx)\bar{\rho}_s(dy)ds\leq \tilde{C}_{int},
\end{align*}
and by monotone convergence theorem
\begin{align*}
\mathbb{E}\int_{0}^t\int\int\frac{\mathds{1}_{x\neq y}}{|x-y|^{\alpha-1}}\bar{\rho}_s(dx)\bar{\rho}_s(dy)ds\leq \tilde{C}_{int}.
\end{align*}
This implies
\begin{align}
\lim_{R\rightarrow 0}\lim_{N\rightarrow \infty}\mathbb{E}&\int_{0}^t\int\int\frac{(f'(x)-f'(y))(x-y)}{|x-y|^{\alpha+1}}(1-\phi_R(|x-y|))\mu^N_s(dx)\mu^N_s(dy)ds\nonumber\\
&=\int_{0}^t\int\int_{\{x\neq y\}}\frac{(f'(x)-f'(y))(x-y)}{|x-y|^{\alpha+1}}\bar{\rho}_s(dx)\bar{\rho}_s(dy)ds\label{I_4_terme2}.
\end{align}
From \eqref{I_4_terme1} and \eqref{I_4_terme2}, we obtain
\begin{align*}
\mathbb{E}I_4(N)\longrightarrow\frac{1}{2}\int_{0}^t\int\int_{\{x\neq y\}}\frac{(f'(x)-f'(y))(x-y)}{|x-y|^{\alpha+1}}\bar{\rho}_s(dx)\bar{\rho}_s(dy)ds.
\end{align*}
\end{description}
Hence the result.
\begin{flushright} $\openbox$ \end{flushright}


\begin{remark}
Notice how above we rely on the fact that $\frac{(f'(x)-f'(y))(x-y)}{|x-y|^{\alpha+1}}\mathds{1}_{x\neq y}$ is integrable with respect to $\bar{\rho}_t\otimes\bar{\rho}_t$ for a Lipschitz continuous function $f'$. This amounts to being able to prove
\begin{align*}
\int\int_{\{x\neq y\}}\frac{1}{|x-y|^{\alpha-1}}\bar{\rho}_t(dx)\bar{\rho}_t(dy)<\infty.
\end{align*}
For the sake of the argument, let us assume that $\bar{\rho}_t=\mathds{1}_{[0,1]}$ is the uniform distribution on $[0,1]$. Then 
\begin{align*}
\int\int_{[0,1]\times[0,1]}\frac{1}{|x-y|^{\alpha-1}}dxdy<\infty,
\end{align*}
if and only if $\alpha<2$. Although this is no proof, this small estimate seems to indicate that $\alpha=2$ is indeed a critical value.
\end{remark}

%
%
%
%

\section{From weak propagation of chaos to strong uniform in time propagation of chaos}\label{sec_from_weak_to_strong}

In this section, we wish to show how one could improve a result of weak propagation of chaos, as for instance obtained in \cite{Rogers_Shi}, \cite{cepa_lepingle} or \cite{Li_Li_Xie}, into a result of strong and uniform in time propagation of chaos. We consider \eqref{particle_system} for any potentials $U$ and $V$ and any diffusion $\sigma_N$, and assume there is a strong solution $(X^i_t)_t$ of \eqref{particle_system}. In this general framework, we assume one has been able to prove the following assertions :

\begin{assumption}\label{Hyp_weak_prop}[Weak prop. of chaos]
For an initial distribution $\mu_0$ converging in $L^2$ Wasserstein distance to a measure $\bar{\rho}_0$, and for all $t\geq0$, the empirical measure ${\mu^N_t=\frac{1}{N}\sum_{i=1}^N\delta_{X^i_t}}$ converges weakly to a probability density $\bar{\rho}_t$.
\end{assumption}

\begin{assumption}\label{Hyp_bounded_moments}[Bounded moments]
Assume there is $C_0\geq0$ such that for all $N$ and all $t\geq0$
$$\mathbb{E}\left(\frac{1}{N}\sum_{i=1}^N|X^i_t|^{4}\right)\leq C_0.$$
\end{assumption}

\begin{assumption}\label{Hyp_long_time_cv}[Long time convergence]
Denoting by $\rho^{1,N}_t$ and $\rho^{2,N}_t$ the probability densities of the $N$ particle systems in $\mathcal{O}_N$ with respective initial conditions $\rho^{1,N}_0$ and $\rho^{2,N}_0$, there exists $\lambda>0$ such that we have
\begin{align*}
\forall t\geq0, \mathcal{W}_2\left(\rho^{1,N}_t,\rho^{2,N}_t\right)\leq e^{-\lambda t}\mathcal{W}_2\left(\rho^{1,N}_0,\rho^{2,N}_0\right)
\end{align*}.
\end{assumption}

\begin{assumption}\label{Hyp_cont}[Continuity in 0]
The function $t\mapsto\mathbb{E}\left(\frac{1}{N}\sum_{i=1}^N|X^i_t-X^i_0|^2\right)$ is continuous in $t=0$, uniformly in $N$, in the sense that
\begin{align*}
\forall \epsilon>0,\ \exists \delta>0,\ \forall 0\leq t<\delta,\ \forall N\geq0,\ \mathbb{E}\left(\frac{1}{N}\sum_{i=1}^N|X^i_t-X^i_0|^2\right)\leq\epsilon
\end{align*}
\end{assumption}

\begin{remark}
These assumptions are satisfied in the case $\alpha=1$. We have shown in Theorem~\ref{long_time_beha} the long time convergence of the particle system. In Appendix~\ref{cont_0_alpha_1} we prove continuity in 0 for a well chosen initial condition. To prove the bounded 4-th moments, considering $\phi:(x_1,..,x_N)\mapsto \frac{1}{N}\sum_{i=1}^N|x_i|^4$, we have
\begin{align*}
\mathcal{L}^{N,\alpha}\phi=&-\sum_{i=1}^N(\lambda x_i)\left(\frac{4}{N}x_i^3\right)+\sum_{i=1}^N\left(\frac{1}{N}\sum_{j\neq i}^N\frac{1}{x_i-x_j}\right)\left(\frac{4x_i^3}{N}\right)+\sigma_N\sum_{i=1}^N \frac{12}{N}x_i^2\\
=&12\sigma_N\left(\frac{1}{N}\sum_{i=1}^N x_i^2\right)-\frac{4\lambda}{N}\sum_{i=1}^N|x_i|^4+\frac{4}{N^2}\sum_{i\neq j}\frac{x_i^3}{x_i-x_j}.
\end{align*}
We get
\begin{align*}
\frac{4}{N^2}\sum_{i\neq j}\frac{x_i^3}{x_i-x_j}&=\frac{4}{N^2}\sum_{j< i}\frac{x_i^3-x_j^3}{x_i-x_j}=\frac{4}{N^2}\sum_{j< i}x_i^2+x_ix_j+x_j^2\\
&\leq\frac{6}{N^2}\sum_{j< i}x_i^2+x_j^2\leq\frac{6}{N}\sum_{i=1}^Nx_i^2.
\end{align*}
This way
\begin{align*}
\mathcal{L}^{N,\alpha}\phi\leq &6\left(2\sigma_N+1\right)\left(\frac{1}{N}\sum_{i=1}^N x_i^2\right)-\frac{4\lambda}{N}\sum_{i=1}^N|x_i|^4\\
\leq&\frac{9\left(2\sigma_N+1\right)^2}{2\lambda}-\frac{2\lambda}{N}\sum_{i=1}^N|x_i|^4\ \ \ \ \text{ since }\ \ \ \ x_i^2\leq \frac{\lambda x_i^4}{3(1+2\sigma_N)}+\frac{3(1+2\sigma_N)}{4\lambda }\\
=&\frac{9\left(2\sigma_N+1\right)^2}{2\lambda}-2\lambda\phi.
\end{align*}
Hence the uniform in time bound on the 4-th moment provided we have an initial bound.
\end{remark}

Our goal is to show
\begin{theorem}\label{unif_prop_chaos_a_partir_de_tout_un_tas_de_proprietes_que_je_ne_vais_pas_donner_parce_que_la_reference_devient_trop_longue_mais_vous_voyez_l_idee}
Under Assumptions~\ref{Hyp_weak_prop}, \ref{Hyp_bounded_moments}, \ref{Hyp_long_time_cv} and \ref{Hyp_cont}, we get strong uniform in time propagation of chaos, i.e
\begin{align*}
\forall \epsilon>0,\ \exists N\geq0,\ \forall t\geq0,\ \forall n\geq N,\ \mathbb{E}\left(\mathcal{W}_2\left(\mu^n_t,\bar{\rho}_t\right)\right)<\epsilon.
\end{align*}
\end{theorem}

The outline of the proof is the following
\begin{itemize}
\item Using the weak propagation of chaos and the bounded moments, we get a strong convergence in Wasserstein distance.
\item Using the long time convergence of the particle system and the strong propagation of chaos, we get the long time convergence for the limiting process, as well as strong propagation of chaos for the stationnary measures.
\item Thanks to the long time convergence of both the particle system and the limiting process, and using the continuity in $0$ of the particle system for the Wasserstein distance, we get uniform continuity in time for the Wasserstein distance between the empirical measure and the limiting process, this continuity being uniform in $N$.
\item Finally, thanks to all the previous results, we get uniform in time propagation of chaos.
\end{itemize}

The following result is the characterization of the $\mathcal{W}_2$-convergence, as given in Theorem~6.9 of \cite{villani2008optimal}. 


\begin{lemma}\label{strong_prop}[Strong propagation of chaos]
Under Assumptions \ref{Hyp_weak_prop} and \ref{Hyp_bounded_moments}, we also have the following convergence
\begin{equation}
\forall t\geq0,\ \lim_{N\rightarrow\infty}\mathbb{E}\left(\mathcal{W}_2\left(\mu^N_t,\bar{\rho}_t\right)^2\right)=0
\end{equation}
\end{lemma}


\begin{remark}
We use here the assumption on the bounded 4-th moment of the empirical measure, to have by Cauchy-Schwarz inequality for $R>0$
\begin{align*}
\frac{1}{N}\sum_{i=1}^N\left|X^i\right|^2\mathds{1}_{\left|X^i\right|\geq \frac{R}{2}}\leq& \left(\frac{1}{N}\sum_{i=1}^N\left|X^i\right|^{4}\right)^{1/2}\left(\frac{1}{N}\sum_{i=1}^N\mathds{1}_{\left|X^i\right|\geq \frac{R}{2}}\right)^{1/2}\\
\mathbb{E}\left(\frac{1}{N}\sum_{i=1}^N\left|X^i\right|^2\mathds{1}_{\left|X^i\right|\geq \frac{R}{2}}\right)\leq& \mathbb{E}\left(\frac{1}{N}\sum_{i=1}^N\left|X^i\right|^{4}\right)^{1/2}\mathbb{E}\left(\frac{1}{N}\sum_{i=1}^N\mathds{1}_{\left|X^i\right|\geq \frac{R}{2}}\right)^{1/2}\\
\leq& C_0^{1/2}\mathbb{E}\left(\frac{1}{N}\sum_{i=1}^N\mathds{1}_{\left|X^i\right|\geq \frac{R}{2}}\right)^{1/2}.
\end{align*}
We have, by weak convergence since $x\mapsto\mathds{1}_{\left|x\right|\geq \frac{R}{2}}$ is a bounded upper semi continuous function
$$\limsup_{N\rightarrow\infty}\int\mathds{1}_{\left|x\right|\geq \frac{R}{2}}d\mu^N_t\leq\int\mathds{1}_{\left|x\right|\geq \frac{R}{2}}d\bar{\rho}_t.$$
Then, since $\int\mathds{1}_{\left|x\right|\geq \frac{R}{2}}d\mu^N_t$ is a sequence of positive functions such that \linebreak $\int\mathds{1}_{\left|x\right|\geq \frac{R}{2}}d\mu^N_t\leq 1$, we have by Fatou's lemma
$$\limsup_{N\rightarrow\infty}\mathbb{E}\left(\int\mathds{1}_{\left|x\right|\geq \frac{R}{2}}d\mu^N_t\right)\leq\mathbb{E}\left(\limsup_{N\rightarrow\infty}\int\mathds{1}_{\left|x\right|\geq \frac{R}{2}}d\mu^N_t\right)$$
and by dominated convergence
$$\lim_{R\rightarrow\infty}\int\mathds{1}_{\left|x\right|\geq \frac{R}{2}}d\bar{\rho}_t=0.$$
Therefore
\begin{align*}
\lim_{R\rightarrow\infty}\limsup_{N\rightarrow\infty}\mathbb{E}\left(\frac{1}{N}\sum_{i=1}^N\mathds{1}_{\left|X^i\right|\geq \frac{R}{2}}\right)\leq&\lim_{R\rightarrow\infty}\mathbb{E}\left(\limsup_{N\rightarrow\infty}\frac{1}{N}\sum_{i=1}^N\mathds{1}_{\left|X^i\right|\geq \frac{R}{2}}\right)\\
\leq&\lim_{R\rightarrow\infty}\mathbb{E}\left(\int\mathds{1}_{\left|x\right|\geq \frac{R}{2}}d\bar{\rho}_t\right)\\
=&\lim_{R\rightarrow\infty}\int\mathds{1}_{\left|x\right|\geq \frac{R}{2}}d\bar{\rho}_t\\
=&0.
\end{align*}
This yields the necessary property to use Theorem~6.9 of \cite{villani2008optimal}. In reality, any assumption on a bounded p-th moment with $p>2$ would have been sufficient, using Hölder's inequality instead of Cauchy-Schwarz's.
\end{remark}


\begin{lemma}[Long time behavior of the limiting equation]
Consider $\mu^N_t$ (resp. $\nu^N_t$) the empirical distribution of the solution $(X^i_t)_t$ with initial distribution $\mu_0^{\otimes N}$ (resp. $\nu_0^{\otimes N}$) and weakly converging as $N$ goes to infinity to $\bar{\mu}_t$ (resp. $\bar{\nu}_t$). Under Assumptions \ref{Hyp_weak_prop},  \ref{Hyp_bounded_moments} and \ref{Hyp_long_time_cv}, we have
\begin{align*}
\mathcal{W}_2\left(\bar{\mu}_t, \bar{\nu}_t\right)\leq e^{-\lambda t}\mathcal{W}_2\left(\bar{\mu}_0, \bar{\nu}_0\right).
\end{align*}
\end{lemma}


\begin{proof}
Denoting $\rho^{1,N}_t$ (resp. $\rho^{2,N}_t$) the law in $\mathcal{O}_N$ the law of the $N$ particle system which yields $\mu_t^N$ (resp. $\nu_t^N$).  We have, under $\pi^N_t$ the optimal coupling between $\rho^{1,N}_t$ and $\rho^{2,N}_t$ for the $L^2$ Wasserstein distance.
\begin{align*}
\mathcal{W}_2\left(\bar{\mu}_t, \bar{\nu}_t\right)\leq\mathbb{E}^{\pi^N_t}\left(\mathcal{W}_2\left(\bar{\mu}_t, \mu^N_t\right)+\mathcal{W}_2\left(\mu^N_t, \nu^N_t\right)+\mathcal{W}_2\left(\nu^N_t, \bar{\nu}_t\right)\right).
\end{align*}
Since 
\begin{align*}
\mathbb{E}^{\pi^N_t}\left(\mathcal{W}_2\left(\mu^N_t, \nu^N_t\right)\right)\leq
\mathbb{E}^{\pi^N_t}\left(\mathcal{W}_2\left(\mu^N_t, \nu^N_t\right)^2\right)^{1/2}&=\mathbb{E}^{\pi^N_t}\left(\frac{1}{N}\sum_{i=1}^N(X^i_t-Y^i_t)^2\right)^{1/2}\\
&=\frac{1}{\sqrt{N}}\mathcal{W}_2(\rho^{1,N}_t,\rho^{2,N}_t),
\end{align*}
and
\begin{align*}
\mathcal{W}_2(\rho^{1,N}_t,\rho^{2,N}_t)\leq e^{-\lambda t}\mathcal{W}_2(\rho^{1,N}_0,\rho^{2,N}_0)&\leq e^{-\lambda t}\mathbb{E}\left(\sum_{i=1}^N(X^i_0-Y^i_0)^2\right)^{1/2}\\
&=\sqrt{N}e^{-\lambda t}\mathbb{E}\left(\mathcal{W}_2\left(\mu^N_0, \nu^N_0\right)\right),
\end{align*}
(where this last expectation is taken for any coupling of $\rho^{1,N}_0$ and $\rho^{2,N}_0$) we get, for all $N\geq0$
\begin{align*}
\mathcal{W}_2\left(\bar{\mu}_t, \bar{\nu}_t\right)\leq&\mathbb{E}\left(\mathcal{W}_2\left(\bar{\mu}_t, \mu^N_t\right)+e^{-\lambda t}\left(\mathcal{W}_2\left(\mu^N_0, \bar{\mu}_0\right)+\mathcal{W}_2\left(\bar{\mu}_0, \bar{\nu}_0\right)+\mathcal{W}_2\left(\bar{\nu}_0, \nu^N_0\right)\right)\right.\\
&\hspace{1cm}\left.+\mathcal{W}_2\left(\nu^N_t, \bar{\nu}_t\right)\right).
\end{align*}
Recall from Lemma~\ref{strong_prop} ${\mathbb{E}\left(\mathcal{W}_2\left(\bar{\mu}_t, \mu^N_t\right)\right)\rightarrow0}$ as $N$ tends to infinity. By taking the limit as $N$ tends to infinity in the righthand side of the inequality above , we obtain
\begin{align*}
\mathcal{W}_2\left(\bar{\mu}_t, \bar{\nu}_t\right)\leq e^{-\lambda t}\mathcal{W}_2\left(\bar{\mu}_0, \bar{\nu}_0\right).
\end{align*}
\end{proof}

The contraction of the Wasserstein distance for the non linear limit yields the existence of a stationary distribution. 


\begin{lemma}[Propagation of chaos for the stationary distribution]
Under Assumptions~\ref{Hyp_weak_prop}, \ref{Hyp_bounded_moments}, \ref{Hyp_long_time_cv} and \ref{Hyp_cont}, denote by $\bar{\rho}_{\infty}$ (resp. $\rho^N_{\infty}$) the stationary measure for the non linear process (resp. for the particle system), and let $\mu^N_{\infty}$ be an empirical measure associated to $\rho^N_{\infty}$. We have
\begin{align*}
\mathbb{E}\left(\mathcal{W}_2\left(\mu^N_{\infty},\bar{\rho}_{\infty}\right)^2\right)\rightarrow 0\ \ \ \ \text{ as }\ \ \ \ N\rightarrow\infty.
\end{align*}
\end{lemma}


\begin{proof}
We have
\begin{align*}
\mathcal{W}_2\left(\bar{\rho}_{t},\bar{\rho}_{\infty}\right)\leq e^{-\lambda t}\mathcal{W}_2\left(\bar{\rho}_{0},\bar{\rho}_{\infty}\right),\\
\mathcal{W}_2\left(\rho^N_{t},\rho^N_{\infty}\right)\leq e^{-\lambda t}\mathcal{W}_2\left(\rho^N_{0},\rho^N_{\infty}\right).
\end{align*}
Let $\mu^N_{\infty}$ be an empirical measure associated to $\rho^N_{\infty}$. We have, for all $t\geq0$, under $\pi_t$ the optimal coupling between $\rho^N_{t}$ and $\rho^N_{\infty}$
\begin{align*}
\mathbb{E}^{\pi_t}\left(\mathcal{W}_2\left(\mu^N_{\infty},\bar{\rho}_{\infty}\right)^2\right)\leq& 3\mathbb{E}^{\pi_t}\left(\mathcal{W}_2\left(\mu^N_{\infty},\mu^N_{t}\right)^2\right)+3\mathbb{E}^{\pi_t}\left(\mathcal{W}_2\left(\mu^N_{t},\bar{\rho}_{t}\right)^2\right)\\
&+3\mathbb{E}^{\pi_t}\left(\mathcal{W}_2\left(\bar{\rho}_{t},\bar{\rho}_{\infty}\right)^2\right).
\end{align*}
We consider an initial condition $\bar{\rho}_0=\bar{\rho}_\infty$ (and thus for all $t\geq0$, $\bar{\rho}_{t}=\bar{\rho}_{\infty}$) and $X^1_0,...,X^N_0$ i.i.d initial condition (reordered) distributed according to $\bar{\rho}_\infty$ (this way $\mathbb{E}(\mathcal{W}_2(\mu_0^N,\bar{\rho}_\infty))\rightarrow 0$). We get 
\begin{align*}
\mathbb{E}^{\pi_t}\left(\mathcal{W}_2\left(\mu^N_{\infty},\bar{\rho}_{\infty}\right)^2\right)\leq 3\mathbb{E}^{\pi_t}\left(\mathcal{W}_2\left(\mu^N_{\infty},\mu^N_{t}\right)^2\right)+3\mathbb{E}^{\pi_t}\left(\mathcal{W}_2\left(\mu^N_{t},\bar{\rho}_{t}\right)^2\right).
\end{align*}
On one hand, since the optimal transport map for the $\mathcal{W}_2$ distance between two sets of points in dimension one is the map that transports the first point to the first point, the second to the the second, etc, when the two sets are ordered,
\begin{align}
\mathbb{E}^{\pi_t}\left(\mathcal{W}_2\left(\mu^N_{\infty},\mu^N_{t}\right)^2\right)&=\mathbb{E}^{\pi_t}\left(\frac{1}{N}\sum_{i=1}^N(X^i-Y^i)^2\right)
=\frac{1}{N}\mathcal{W}_2\left(\rho^N_{t},\rho^N_{\infty}\right)^2
\nonumber\\
&\leq\frac{e^{-2\lambda t}}{N}\mathcal{W}_2\left(\rho^N_{0},\rho^N_{\infty}\right)^2.\label{maj_mes_emp_dens}
\end{align}
Then, there exists a constant $C_0$, depending on the uniform bounds on the second moments of the non linear process and the empirical measure of the particle system such that
\begin{align*}
\mathbb{E}^{\pi_t}\left(\mathcal{W}_2\left(\mu^N_{\infty},\mu^N_{t}\right)^2\right)\leq C_0e^{-2\lambda t}
\end{align*}
On the second hand, as $\bar{\rho}_{t}$ is a deterministic measure, we have
\begin{align*}
\mathbb{E}^{\pi_t}\left(\mathcal{W}_2\left(\mu^N_{t},\bar{\rho}_{t}\right)^2\right)=\mathbb{E}\left(\mathcal{W}_2\left(\mu^N_{t},\bar{\rho}_{t}\right)^2\right)\rightarrow 0\ \ \ \ \text{ as }\ \ \ \ N\rightarrow\infty.
\end{align*}
This yields
\begin{align*}
\mathbb{E}\left(\mathcal{W}_2\left(\mu^N_{\infty},\bar{\rho}_{\infty}\right)^2\right)\leq C_0e^{-2\lambda t}+3\mathbb{E}\left(\mathcal{W}_2\left(\mu^N_{t},\bar{\rho}_{t}\right)^2\right).
\end{align*}
Consider $\epsilon>0$. There is $t_\epsilon$ such that for all $t\geq t_\epsilon$ we have $C_0e^{-2\lambda t}\leq\frac{\epsilon}{2}$ and, given $t_\epsilon$, there is a $N_\epsilon$ such that for all $N\geq N_\epsilon$ we have ${3\mathbb{E}\left(\mathcal{W}_2\left(\mu^{N}_{t_\epsilon},\bar{\rho}_{t_\epsilon}\right)^2\right)\leq\frac{\epsilon}{2}}$. This way
\begin{align*}
\forall \epsilon >0,\ \ \exists N_\epsilon \geq0,\ \ \forall N\geq N_\epsilon,\ \  \mathbb{E}\left(\mathcal{W}_2\left(\mu^N_{\infty},\bar{\rho}_{\infty}\right)^2\right)\leq\epsilon,
\end{align*}
i.e
\begin{align*}
\mathbb{E}\left(\mathcal{W}_2\left(\mu^N_{\infty},\bar{\rho}_{\infty}\right)^2\right)\rightarrow 0\ \ \ \ \text{ as }\ \ \ \ N\rightarrow\infty.
\end{align*}
\end{proof}


\begin{lemma}[Uniform continuity in $t$, uniformly in $N$]\label{cont_unif}
Under Assumptions~\ref{Hyp_weak_prop}, \ref{Hyp_bounded_moments}, \ref{Hyp_long_time_cv} and \ref{Hyp_cont}, the function ${t\rightarrow\mathbb{E}\left(\mathcal{W}_2\left(\mu^N_{t},\bar{\rho}_{t}\right)\right)}$ is uniformly continuous in $t$, uniformly in $N$.
\end{lemma}


\begin{proof}
Let us begin by showing the function $t\rightarrow\mathcal{W}_2\left(\bar{\rho}_{t},\bar{\rho}_{0}\right)$ is continuous in $t=0$. We have, for all $N\geq0$
\begin{align*}
\mathcal{W}_2\left(\bar{\rho}_{t},\bar{\rho}_{0}\right)\leq\mathbb{E}\left(\mathcal{W}_2\left(\bar{\rho}_t,\mu^N_t\right)+\mathcal{W}_2\left(\mu^N_t,\mu^N_0\right)+\mathcal{W}_2\left(\mu^N_0,\bar{\rho}_{0}\right)\right).
\end{align*}
Let $\epsilon>0$.
First, we have
\begin{align*}
\mathbb{E}\left(\mathcal{W}_2\left(\mu^N_t,\mu^N_0\right)^2\right)=\mathbb{E}\left(\frac{1}{N}\sum_{i=1}^N|X^i_t-X^i_0|^2\right),
\end{align*}
and thus, by Assumption~\ref{Hyp_cont}, there exists $\delta>0$ such that for all $t\leq \delta$ and $N\geq0$
\begin{align*}
\mathbb{E}\left(\mathcal{W}_2\left(\mu^N_t,\mu^N_0\right)\right)\leq\frac{\epsilon}{3}.
\end{align*}
Then, let $t\leq\delta$. Using the strong propagation of chaos, there exists $N_t\geq0$ and $N_0\geq0$ such that for $N=\max(N_t,N_0)$
\begin{align*}
\mathbb{E}\left(\mathcal{W}_2\left(\bar{\rho}_t,\mu^{N}_t\right)\right)\leq\frac{\epsilon}{3}\ \ \ \text{ and }\ \ \ \mathbb{E}\left(\mathcal{W}_2\left(\mu^{N}_0,\bar{\rho}_{0}\right)\right)\leq\frac{\epsilon}{3}
\end{align*}
Hence 
$$\forall \epsilon>0,\ \exists\delta>0,\ \forall t<\delta,\ \mathcal{W}_2\left(\bar{\rho}_{t},\bar{\rho}_{0}\right)<\epsilon,$$
and the continuity of the function $t\rightarrow\mathcal{W}_2\left(\bar{\rho}_{t},\bar{\rho}_{0}\right)$ in $t=0$.

Now, let $t\geq0$ and $(t_n)_{n\in\mathbb{N}}$ a sequence converging to $t$. We have
\begin{align*}
&\left|\mathbb{E}^{\pi^N_{t,t_n}}\left(\mathcal{W}_2\left(\mu^N_{t_n},\bar{\rho}_{t_n}\right)\right)-\mathbb{E}^{\pi^N_{t,t_n}}\left(\mathcal{W}_2\left(\mu^N_{t},\bar{\rho}_{t}\right)\right)\right|\\
&\hspace{1cm}\leq \left|\mathbb{E}^{\pi^N_{t,t_n}}\left(\mathcal{W}_2\left(\mu^N_{t_n},\mu^N_{t}\right)\right)+\mathbb{E}^{\pi^N_{t,t_n}}\left(\mathcal{W}_2\left(\bar{\rho}_{t},\bar{\rho}_{t_n}\right)\right)\right|\\
&\hspace{1cm}\leq e^{-\lambda (t\wedge t_n)}\left(\frac{1}{\sqrt{N}}\mathcal{W}_2\left(\rho^N_{|t-t_n|},\rho^N_{0}\right)+\mathcal{W}_2\left(\bar{\rho}_{|t-t_n|},\bar{\rho}_{0}\right)\right),
\end{align*}
where the expectation is taken under $\pi^N_{t,t_n}$ the optimal coupling between $\rho^N_{t_n}$ and $\rho^N_{t}$ and the last inequality comes from the fact that 
\begin{align*}
\mathbb{E}^{\pi^N_{t,t_n}}\left(\mathcal{W}_2\left(\mu^N_{t_n},\mu^N_{t}\right)\right)&\leq\mathbb{E}^{\pi^N_{t,t_n}}\left(\mathcal{W}_2\left(\mu^N_{t_n},\mu^N_{t}\right)^2\right)^{1/2}=\mathbb{E}^{\pi^N_{t,t_n}}\left(\frac{1}{N}\sum_{i=1}^N(X^i-Y^i)^2\right)^{1/2}\\
&=\frac{1}{\sqrt{N}}\mathcal{W}_2\left(\rho^N_{t},\rho^N_{t_n}\right).
\end{align*}
We have 
\begin{align*}
\frac{1}{N}\mathcal{W}_2\left(\rho^N_{|t-t_n|},\rho^N_{0}\right)^2\leq \mathbb{E}\left(\frac{1}{N}\sum_{i=1}^N|X^i_{|t-t_n|}-X^i_0|^2\right).
\end{align*}
The continuity in $0$ of ${t\rightarrow\mathcal{W}_2\left(\bar{\rho}_{t},\bar{\rho}_{0}\right)}$, and the continuity in $0$  (uniform in $N$) of ${t\rightarrow\mathbb{E}\left(\frac{1}{N}\sum_{i=1}^N|X^i_t-X^i_0|^2\right)}$ are therefore sufficient to yield the result.
\end{proof}


\begin{lemma}\label{t_N_def}
Under Assumptions~\ref{Hyp_weak_prop}, \ref{Hyp_bounded_moments}, \ref{Hyp_long_time_cv} and \ref{Hyp_cont}, there exists a non-decreasing sequence $(t_N)_{N\geq0}$ that goes to infinity such that for all $N\geq0$
\begin{equation}\label{Prop_1}
\sup_{s\in[0,t_N]}\mathbb{E}\left(\mathcal{W}_2\left(\mu^N_{s},\bar{\rho}_{s}\right)\right)\rightarrow 0\ \ \ \ \text{ as }\ \ \ \ N\rightarrow\infty.
\end{equation}
\end{lemma}


\begin{proof}
By strong propagation of chaos, 
\begin{equation}\label{Conv_non_unif_base}
\forall \epsilon>0,\ \ \forall t\geq0,\ \ \exists N\geq0,\ \ \forall n\geq N,\ \ \mathbb{E}\left(\mathcal{W}_2\left(\mu^n_{t},\bar{\rho}_{t}\right)\right)\leq \epsilon.
\end{equation}
Denote $g(t,N)=\mathbb{E}\left(\mathcal{W}_2\left(\mu^N_{t},\bar{\rho}_{t}\right)\right)$. By Lemma~\ref{cont_unif}, $g$ is uniformly continuous in $t$, uniformly in $N$. Let $\epsilon>0$ and $t>0$. There exists $N_1\geq0$ such that for all $n\in\mathbb{N}$ and all $x,y\in[0,t]$
\begin{align*}
|x-y|\leq\frac{t}{N_1}\implies |g(x,n)-g(y,n)|\leq\frac{\epsilon}{2}.
\end{align*}
We also have
\begin{align*}
\exists N\geq0,\ \ \forall n\geq N,\ \ \forall i\in \{0,..,N_1\},\ \  g\left(t\frac{i}{N_1},n\right)\leq \frac{\epsilon}{2}.
\end{align*}
This way
\begin{align*}
\exists N\geq0,\ \ \forall n\geq N,\ \ \forall s\in [0,t],\ \  g(s,n)\leq \epsilon.
\end{align*}
Denoting $f(t,N)=\sup_{s\in[0,t]}\mathbb{E}\left(\mathcal{W}_2\left(\mu^N_{s},\bar{\rho}_{s}\right)\right)$, we thus obtain
\begin{equation}\label{Conv_non_unif_sup}
\forall \epsilon>0,\ \ \forall t\geq0,\ \ \exists N\geq0,\ \ \forall n\geq N,\ \ f(t,n)\leq \epsilon.
\end{equation}
There exists a non-decreasing function $\phi:\mathbb{R}\mapsto \mathbb{N}$ such that for all $t\geq0$ and all $n\geq \phi(t)$ we have $f(t,n)\leq\frac{1}{t}$ and $\lim_{t\rightarrow\infty}\phi(t)=+\infty$. By convention $\phi(0)=0$.

Consider $t_0=0$ and 
$$ t_N=\sup\{t\geq t_{N-1}\text{ s.t. }t\in\phi^{-1}(\{0,1,..,N\})\}.$$
The sequence $(t_N)_{N\geq0}$ thus defined is non-decreasing by construction. Because $\lim_{t\rightarrow\infty}\phi(t)=+\infty$, the set $\phi^{-1}(\{0,1,..,N\})$ is non-empty and its supremum goes to infinity as $N$  goes to infinity. Therefore $\lim_{N\rightarrow\infty}t_N=+\infty$ and $t_N\neq0$ eventually.

We have $N\geq\phi(t_{N-1})$, and therefore by definition of $\phi$, we eventually get for $N$ sufficiently large
$$f(t_{N-1},N)\leq\frac{1}{t_{N-1}}.$$
This concludes the proof.
\end{proof}

We may now conclude.

\begin{proof}[Proof of Theorem~\ref{unif_prop_chaos_a_partir_de_tout_un_tas_de_proprietes_que_je_ne_vais_pas_donner_parce_que_la_reference_devient_trop_longue_mais_vous_voyez_l_idee}]
We have, for $\mu_\infty^N$ an empirical measure associated to $\rho^N_\infty$, and $\pi^N_t$ the optimal coupling between $\rho^N_t$ and $\rho^N_\infty$
\begin{align*}
\mathbb{E}^{\pi^N_t}\left(\mathcal{W}_2\left(\mu^N_{t},\bar{\rho}_{t}\right)\right)\leq& \mathbb{E}^{\pi^N_t}\left(\mathcal{W}_2\left(\mu^N_{t},\mu^N_{\infty}\right)\right)+\mathbb{E}^{\pi^N_t}\left(\mathcal{W}_2\left(\mu^N_{\infty},\bar{\rho}_{\infty}\right)\right)+\mathbb{E}^{\pi^N_t}\left(\mathcal{W}_2\left(\bar{\rho}_{\infty},\bar{\rho}_{t}\right)\right)\\
\leq&e^{-\lambda t}\left(\frac{1}{\sqrt{N}}\mathcal{W}_2\left(\rho^N_0,\rho^N_{\infty}\right)+\mathcal{W}_2\left(\bar{\rho}_{\infty},\bar{\rho}_{0}\right)\right)+\mathbb{E}\left(\mathcal{W}_2\left(\mu^N_{\infty},\bar{\rho}_{\infty}\right)\right).
\end{align*}
Since 
\begin{align*}
\frac{1}{\sqrt{N}}\mathcal{W}_2\left(\rho^N_0,\rho^N_{\infty}\right)\leq \mathbb{E}\left(\frac{1}{N}\sum_{i=1}^N|X^i_0-X^i_\infty|^2\right)^{1/2}\leq 2 C_0^{1/2},
\end{align*}
we obtain
\begin{equation}\label{Prop_2}
\mathbb{E}\left(\mathcal{W}_2\left(\mu^N_{t},\bar{\rho}_{t}\right)\right)\leq \tilde{C}(t)+\tilde{f}(N),
\end{equation}
where $\tilde{C}$ is decreasing and goes to 0, and $\tilde{f}$ tends to 0. Using \eqref{Prop_1} and \eqref{Prop_2}, we get
\begin{align*}
\mathbb{E}\left(\mathcal{W}_2\left(\mu^N_{t},\bar{\rho}_{t}\right)\right)\leq \min\left(\sup_{s\in[0,t]}\mathbb{E}\left(\mathcal{W}_2\left(\mu^N_{s},\bar{\rho}_{s}\right)\right), \tilde{C}(t)+\tilde{f}(N)\right).
\end{align*}
Let $t\geq0$. If $t\leq t_N$ where $t_N$ is given in Lemma~\ref{t_N_def}, we have 
\begin{align*}
\mathbb{E}\left(\mathcal{W}_2\left(\mu^N_{t},\bar{\rho}_{t}\right)\right)\leq \sup_{s\in[0,t_N]}\mathbb{E}\left(\mathcal{W}_2\left(\mu^N_{s},\bar{\rho}_{s}\right)\right)\rightarrow 0\ \ \ \ \text{ as }\ \ \ \ N\rightarrow\infty,
\end{align*}
and if $t>t_N$
\begin{align*}
\mathbb{E}\left(\mathcal{W}_2\left(\mu^N_{t},\bar{\rho}_{t}\right)\right)\leq \tilde{C}(t)+\tilde{f}(N)\leq \tilde{C}(t_N)+\tilde{f}(N)\rightarrow 0\ \ \ \ \text{ as }\ \ \ \ N\rightarrow\infty.
\end{align*}
Those two bounds being independent of $t$, we obtain uniform in time propagation of chaos.
\end{proof}

%
%
%
%

\appendix

%
%
%
%

\section{Technical results}


\begin{lemma}\label{comp_serie_int}
We have the following inequality
\begin{align*}
\forall N\geq1,\ \ \ \sum_{i=1}^N\frac{1}{i^{\alpha-1}}\leq\left\{
\begin{array}{ll}
\frac{1}{2-\alpha}N^{2-\alpha}&\text{ if }\alpha\in[1,2[,\\
2\ln N&\text{ if }\alpha=2 \text{ (and $N\geq2$)},\\
1+\frac{1}{\alpha-2}&\text{ if }\alpha>2.
\end{array}
\right.
\end{align*}
\end{lemma}


\begin{proof}
Let $f_\alpha:x\rightarrow\frac{1}{x^{\alpha-1}}$. For $\alpha\geq1$, $f_\alpha$ is a non increasing function on $]0,+\infty[$, and for all $x\in[i,i+1]$, $f_\alpha(i+1)\leq f_\alpha(x)\leq f_\alpha(i)$. This implies
$$f_\alpha(i+1)\leq\int_i^{i+1}f_\alpha(x)dx\leq f_\alpha(i),$$
and thus
\begin{align*}
\sum_{i=1}^Nf_\alpha(i)\leq&\left\{
\begin{array}{ll}
\int_0^N\frac{1}{x^{\alpha-1}}dx&\text{ if }\alpha\in[1,2[,\\
f_\alpha(1)+\int_1^N\frac{1}{x^{\alpha-1}}dx&\text{ if }\alpha\geq2.
\end{array}
\right.
\\
=&\left\{
\begin{array}{ll}
\left[\frac{x^{2-\alpha}}{2-\alpha}\right]_0^N&\text{ if }\alpha\in[1,2[,\\
1+\left[\ln x\right]_1^N&\text{ if }\alpha=2,\\
1+\left[\frac{x^{2-\alpha}}{2-\alpha}\right]_1^N&\text{ if }\alpha>2.
\end{array}
\right.
\end{align*}
Hence
\begin{align*}
\sum_{i=1}^N\frac{1}{i^{\alpha-1}}\leq&\left\{
\begin{array}{ll}
\frac{N^{2-\alpha}}{2-\alpha}&\text{ if }\alpha\in[1,2[,\\
1+\ln N&\text{ if }\alpha=2,\\
1+\frac{1}{\alpha-2}-\frac{1}{(\alpha-2)N^{\alpha-2}}&\text{ if }\alpha>2.
\end{array}
\right.
\end{align*}
\end{proof}


\begin{lemma}\label{rho_0_unif}
Let, for $\textbf{x}=(x_i)_{i\in\{1,...,N\}}$
\begin{align*}
A(\textbf{x})=\left(\sum_{j\neq i}\frac{1}{x_i-x_j}\right)_{1\leq i\leq N}.
\end{align*}
There is a constant $C$ such that for all $N\geq0$ and for the set of points $\textbf{x}=(x_i)_{i\in\{1,...,N\}}$ with $x_i=\frac{i}{N}$, we have $|A(\textbf{x})|\leq CN^{3/2}$.
\end{lemma}


\begin{proof}
Throughout this proof, we denote by $C$ the various universal constants appearing for the sake of conciseness. We have, for all $i$
\begin{align*}
\sum_{j\neq i}\frac{1}{x_i-x_j}=N\sum_{j\neq i}\frac{1}{i-j}=N\sum_{j=1}^{i-1}\frac{1}{|i-j|}-N\sum_{j=i+1}^N\frac{1}{|i-j|}=N\left(\sum_{j=1}^{i-1}\frac{1}{j}-\sum_{j=1}^{N-i}\frac{1}{j}\right),
\end{align*}
and thus
\begin{align*}
\sum_{j\neq i}\frac{1}{x_i-x_j}=\left\{\begin{array}{ll}-N\sum_{j=i}^{N-i}\frac{1}{j}&\text{ if }i\leq\lfloor\frac{N+1}{2}\rfloor\\N\sum_{j=N-i+1}^{i-1}\frac{1}{j}&\text{ if }i\geq\lfloor\frac{N+1}{2}\rfloor.\end{array}\right.
\end{align*}
We obtain
\begin{align*}
|A(\textbf{x})|=\left(N^2\sum_{i=1}^{\lfloor\frac{N+1}{2}\rfloor}\left(\sum_{j=i}^{N-i}\frac{1}{j}\right)^2+N^2\sum_{i=1+\lfloor\frac{N+1}{2}\rfloor}^N\left(\sum_{j=N-i+1}^{i-1}\frac{1}{j}\right)^2\right)^{1/2}.
\end{align*}
The change of variable $\tilde{i}=N+1-i$ in this last sum yields
\begin{align*}
|A(\textbf{x})|=\left(2N^2\sum_{i=1}^{\lfloor\frac{N+1}{2}\rfloor}\left(\sum_{j=i}^{N-i}\frac{1}{j}\right)^2\right)^{1/2}.
\end{align*}
There exists a universal constant $C$ such that for all $n\geq0$
\begin{align*}
\ln(n)-C\leq\sum_{i=1}^n\frac{1}{i}\leq \ln(n)+C.
\end{align*}
This yields, for $i\geq2$
\begin{align*}
\left(\sum_{j=i}^{N-i}\frac{1}{j}\right)^2\leq\left(\ln(N-i)-\ln(i-1)+2C\right)^2,
\end{align*}
and for $i=1$
\begin{align*}
\left(\sum_{j=i}^{N-i}\frac{1}{j}\right)^2\leq\left(\ln(N-i)+C\right)^2
\end{align*}
This way
\begin{align*}
|A(\textbf{x})|\leq\left(2N^2\left(\ln(N-i)+C\right)^2+2N^2\sum_{i=2}^{\lfloor\frac{N+1}{2}\rfloor}\left(2(\ln(N-i)-\ln(i-1))^2+8C^2\right)\right)^{1/2}.
\end{align*}
Then
\begin{align*}
(\ln(N-i)-\ln(i-1))^2=\ln\left(\frac{1-\frac{i}{N}}{\frac{i-1}{N}}\right)^2\leq 2\ln\left(1-\frac{i}{N}\right)^2+2\ln\left(\frac{i-1}{N}\right)^2,
\end{align*}
and there is a universal constant, which we also denote by $C$, such that
\begin{align*}
\frac{1}{N}\sum_{i=2}^{\lfloor\frac{N+1}{2}\rfloor}\ln\left(1-\frac{i}{N}\right)^2\leq C+\int_{0}^{1/2}\ln(1-x)^2dx\leq C,
\end{align*}
and
\begin{align*}
\frac{1}{N}\sum_{i=2}^{\lfloor\frac{N+1}{2}\rfloor}\ln\left(\frac{i-1}{N}\right)^2\leq C+\int_{0}^{1/2}\ln(x)^2dx\leq C.
\end{align*}
And thus
\begin{align*}
|A(\textbf{x})|\leq\left(CN^3\right)^{1/2},
\end{align*}
hence the result.
\end{proof}

%
%
%
%

\section{Establishing the continuity in time}\label{cont_0_alpha_1}

In this section, we show the continuity in 0, uniform in $N$, of $t\mapsto\mathbb{E}\left(\frac{1}{N}\sum_{i=1}^N|X^i_t-X^i_0|^2\right)$ (where we denote $\textbf{X}_t=(X^1_t,...,X^N_t)$ the solution of \eqref{particle_system} with initial condition $\textbf{X}_0=(X^1_0,...,X^N_0)\in\mathcal{O}_N$), under some assumptions on $\textbf{X}_0$.


\begin{lemma}\label{mod_cont}
Under Assumptions~\ref{Hyp_U_conv}~and~\ref{Hyp_V}, let $\textbf{X}_t=(X^1_t,...,X^N_t)$ be the solution of \eqref{particle_system} with (deterministic) initial condition $\textbf{x}_0=(x^1_0,...,x^N_0)\in\mathcal{O}_N$. For $\alpha\in[1,2[$, there exists a constant $C_{cont}$ (depending only on $\alpha$ and $\lambda$) such that for all $t\geq0$ and $N\in\mathbb{N}$
\begin{align*}
\mathbb{E}&\left(\frac{1}{N}\sum_{i=1}^N|X^i_t-x^i_0|^2\right)\\
&\leq C_{cont}\left(\frac{|A(\textbf{x}_0)|\mathcal{H}(\textbf{x}_0)}{N^{5/2}}+\frac{(1+\sigma_N)|A(\textbf{x}_0)|}{N^{3/2}}+\frac{\mathcal{H}(\textbf{x}_0)}{N}+1+\sigma_N\right)t,
\end{align*}
where
\begin{equation}\label{def_A}
A(\textbf{x})=\left(-\sum_{j\neq i}V'(x^i-x^j)\right)_{1\leq i\leq N}.
\end{equation}
\end{lemma}


\begin{proof}
Itô's formula yields
\begin{align}
\sum_{i}&|X^i_t-x^i_0|^2\nonumber\\
=&-2\lambda\int_0^t\sum_iX^i_s(X^i_s-x^i_0)ds-\frac{2}{N}\int_0^t\sum_i\left(\sum_{j\neq i}V'(X^i_s-X^j_s)\right)(X^i_s-x^i_0)ds\nonumber\\
&+2N\sigma_Nt+2\sqrt{2\sigma_N}\sum_i\int_0^t(X^i_s-x^i_0)dB^i_s\label{int_lemme_cont}.
\end{align}
We have, using the convexity of $A$,
\begin{align*}
-\sum_i\left(\sum_{j\neq i}V'(X^i_s-X^j_s)\right)(X^i_s-x^i_0)=A(\textbf{X}_s)\cdot(\textbf{X}_s-\textbf{x}_0)\leq |A(\textbf{x}_0)|\left|\textbf{X}_s-\textbf{x}_0\right|,
\end{align*}
and thus
\begin{align*}
\mathbb{E}&\left(-\frac{2}{N}\int_0^t\sum_i\left(\sum_{j\neq i}V'(X^i_s-X^j_s)\right)(X^i_s-x^i_0)ds\right)\\
&\leq\frac{2|A(\textbf{x}_0)|}{N}\mathbb{E}\left(\int_0^t\left(\sum_i(X^i_s-x^i_0)^2\right)^{1/2}ds\right).
\end{align*}
Then
\begin{align*}
\left(\sum_i(X^i_s-x^i_0)^2\right)^{1/2}&\leq\frac{\sqrt{N}}{2}+\frac{1}{2\sqrt{N}}\sum_i(X^i_s-x^i_0)^2\\
&\leq\frac{\sqrt{N}}{2}+\frac{1}{\sqrt{N}}\left(\sum_i(X^i_s)^2+\sum_i(x^i_0)^2\right),
\end{align*}
and thus, using \eqref{borne_H}
\begin{align*}
\left(\sum_i(X^i_s-x^i_0)^2\right)^{1/2}\leq\frac{\sqrt{N}}{2}+4\sqrt{N}+\frac{2}{\sqrt{N}}\left(\mathcal{H}(\textbf{X}_s)+\mathcal{H}(\textbf{x}_0)\right).
\end{align*}
This way
\begin{align*}
\mathbb{E}\left(\int_0^t\left(\sum_i(X^i_s-x^i_0)^2\right)^{1/2}ds\right)\leq \frac{9}{2}\sqrt{N}t+\frac{2}{\sqrt{N}}\mathcal{H}(\textbf{x}_0)t+\frac{2}{\sqrt{N}}\int_0^t\mathbb{E}\left(\mathcal{H}(\textbf{X}_s)\right)ds.
\end{align*}
We now use \eqref{unif_tps_H} to get that there exists a universal constant $C$, depending only on $\alpha$, such that
\begin{align*}
\mathbb{E}\left(\int_0^t\left(\sum_i(X^i_s-x^i_0)^2\right)^{1/2}ds\right)\leq \left(\frac{9}{2}\sqrt{N}+\frac{4}{\sqrt{N}}\mathcal{H}(\textbf{x}_0)+\frac{2\sqrt{N}\sigma_N}{\lambda}+\frac{2C\sqrt{N}}{\lambda}\right)t, 
\end{align*}
which finally yields
\begin{align*}
\mathbb{E}&\left(-\frac{2}{N}\int_0^t\sum_i\left(\sum_{j\neq i}V'(X^i_s-X^j_s)\right)(X^i_s-x^i_0)ds\right)\\
&\hspace{2cm}\leq \left(\frac{9|A(\textbf{x}_0)|}{\sqrt{N}}+\frac{8|A(\textbf{x}_0)|\mathcal{H}(\textbf{x}_0)}{N^{3/2}}+\frac{4|A(\textbf{x}_0)|\sigma_N}{\lambda\sqrt{N}}+\frac{2C|A(\textbf{x}_0)|}{\lambda \sqrt{N}}\right)t.
\end{align*}
We then have, using again \eqref{borne_H}
\begin{align*}
-2\lambda\int_0^t\sum_iX^i_s(X^i_s-x^i_0)ds=&-2\lambda\int_0^t\sum_i(X^i_s)^2ds+2\lambda\int_0^t\sum_i(X^i_s)(x^i_0)ds\\
\leq&-\lambda\int_0^t\sum_i(X^i_s)^2ds+\lambda\int_0^t\sum_i(x^i_0)^2ds\\
\leq&2\lambda(\mathcal{H}(\textbf{x}_0)+N)t.
\end{align*}
and thus 
\begin{align*}
\mathbb{E}\left(-2\lambda\int_0^t\sum_iX^i_s(X^i_s-x^i_0)ds\right)\leq 2\lambda(\mathcal{H}(\textbf{x}_0)+N)t
\end{align*}
Going back to \eqref{int_lemme_cont}, we obtain
\begin{align*}
&\mathbb{E}\left(\frac{1}{N}\sum_{i}|X^i_t-x^i_0|^2\right)\\
&\leq 2\lambda\left(\frac{\mathcal{H}(\textbf{x}_0)}{N}+1\right)t+2\sigma_N t\\
&\hspace{0.3cm}+\left(\frac{9|A(\textbf{x}_0)|}{N^{3/2}}+\frac{8|A(\textbf{x}_0)|\mathcal{H}(\textbf{x}_0)}{N^{5/2}}+\frac{4|A(\textbf{x}_0)|\sigma_N}{\lambda N^{3/2}}+\frac{2C|A(\textbf{x}_0)|}{\lambda N^{3/2}}\right)t,
\end{align*}
hence
\begin{align*}
&\mathbb{E}\left(\frac{1}{N}\sum_{i}|X^i_t-x^i_0|^2\right)\\
&\leq \left(\frac{|A(\textbf{x}_0)|}{N^{3/2}}\left(9+\frac{4\sigma_N}{\lambda}+\frac{2C}{\lambda}\right)+\frac{2\lambda\mathcal{H}(\textbf{x}_0)}{N}+\frac{8|A(\textbf{x}_0)|\mathcal{H}(\textbf{x}_0)}{N^{5/2}}+2\lambda+2\sigma_N\right)t.
\end{align*}
This yields the result.
\end{proof}


\begin{remark}
We thus have to assume the initial condition $\textbf{X}_0=\left(X^i_0\right)_i$ is such that $|A(X_0)|\lesssim N^{3/2}$ and $\mathcal{H}(\textbf{X}_0)\lesssim N$, and still satisfies $\mathcal{W}_2(\mu^N_0,\bar{\rho}_0)\rightarrow0$ as $N\rightarrow\infty$. As shown in Lemma~\ref{rho_0_unif}, for $\alpha=1$ and $\rho_0=\mathds{1}_{[0,1]}$, such a choice is possible.
\end{remark}

%
%

\section*{Acknowledgements}
This work has been (partially) supported by the Project EFI ANR-17-CE40-0030 of the French National Research Agency.

\bibliographystyle{alpha}
\bibliography{GLBM_arxiv_dim1}

\end{document}